\documentclass[11pt,bull-l]{amsproc}
\usepackage[top = 0.9in, bottom = 0.9in, left =1in, right = 1in]{geometry}
\usepackage{geometry}
\usepackage[utf8]{inputenc}
\usepackage{hyperref}
\hypersetup{
  colorlinks   = true, 
  urlcolor     = red, 
  linkcolor    = blue, 
  citecolor   = blue 
}

\usepackage{listings}
\usepackage{multimedia} 
\usepackage{graphicx}
\usepackage{comment}
\usepackage[english]{babel}
\usepackage{amsfonts,amscd,amssymb,amsmath,mathrsfs}
\usepackage{wrapfig}
\usepackage{multirow}
\usepackage{verbatim,cite}
\usepackage{float}
\usepackage{cancel}
\usepackage{caption}
\usepackage{subcaption}
\usepackage{mathdots}
\usepackage{bbm}
\usepackage{tikz}
\usepackage{xcolor}
\usepackage[algoruled]{algorithm2e}
\usepackage{paper-defs}

\title{The Akhiezer iteration}
\author{Cade Ballew}
\address{University of Washington, Seattle, WA}
\email{ballew@uw.edu}

\author{Thomas Trogdon}
\address{University of Washington, Seattle, WA}
\email{trogdon@uw.edu}
\date{}

\begin{document}

\maketitle

\begin{abstract}
We develop the Akhiezer iteration, a generalization of the classical Chebyshev iteration, for the inner product-free, iterative solution of indefinite linear systems using orthogonal polynomials for measures supported on multiple, disjoint intervals. The iteration applies to shifted linear solves and can then be used for efficient matrix function approximation.  Using the asymptotics of orthogonal polynomials, error bounds are provided. A key component in the efficiency of the method is the ability to compute the first $k$ orthogonal polynomial recurrence coefficients and the first $k$ weighted Stieltjes transforms of these orthogonal polynomials in $\OO(k)$ complexity using a numerical Riemann--Hilbert approach.  For a special class of orthogonal polynomials, the Akhiezer polynomials, the method can be sped up significantly, with the greatest speedup occurring in the two interval case where important formulae of Akhiezer are employed and the Riemann--Hilbert approach is bypassed.
\end{abstract}

\section{Introduction}
The Chebyshev iteration \cite{cheby_iter1950,gutknecht_chebyshev_2002, Manteuffel1977} is an inner product-free iterative method often used to solve symmetric, definite linear systems $\bA\bx=\bb$ where an estimate on the eigenvalues of $\bA$ is known. This iteration is based on the minimax property of the Chebyshev polynomials. More specifically, let $\hat T_k$ denote the (ortho)normalized $k$th Chebyshev polynomial shifted and scaled to an interval $[a,b]$ that contains the spectrum of $\bA$, but not zero. Then, the polynomial $r_k(x)=\frac{\hat T_k(x)}{\hat T_k(0)}$ will be small on the spectrum of $\bA$, and $q_k(x)=\frac{1-r_k(x)}{x}$ will satisfy $\bb-\bA q_k(\bA)\bb\approx0$ as well as a convenient recurrence relation.

A somewhat unconventional variant of this method arises from building a Chebyshev approximation to the function $f(x)=\frac{1}{x}$ on an interval containing the spectrum of $\bA$. Specifically, $f$ can be written in a Chebyshev-$\hat T$ series as\footnote{Here $\diamond$ refers to an independent variable, e.g., $1/\diamond$ refers to the function $x \mapsto 1/x$.}
\begin{equation*}
\frac{1}{x}=\sum_{j=0}^{\infty}\left\langle \hat T_j,\frac{1}{\diamond}\right\rangle \hat T_j(x)=\sum_{j=0}^{\infty}\left(\int_{a}^b\frac{1}{z}\frac{\hat T_j(z)\df z}{\pi\sqrt{z-a}\sqrt{b-z}} \right)\hat T_j(x),\quad \int_a^b\frac{\hat T_j(z)\hat T_k(z)}{\pi\sqrt{z-a}\sqrt{b-z}}\df z=\delta_{jk},
\end{equation*}
where $\delta_{jk}$ is the Kronecker delta. This yields a series for the matrix inverse by replacing $\hat T_j(x)$ with the matrix polynomials $\hat T_j(\bA)$. Thus, the solution to the linear system can be approximated as
\begin{equation}\label{eq:cheby_iter}
	\begin{aligned}
	\bx=\bA^{-1}\bb\approx\sum_{j=0}^{k}\left(\int_{a}^b\frac{\hat T_j(z)}{z}\frac{\df z}{\pi\sqrt{z-a}\sqrt{b-z}} \right)\hat T_j(\bA)\bb=:\bx_{k+1}.
	\end{aligned}
\end{equation}
\begin{algorithm}
	\caption{Modified Chebyshev iteration}\label{alg:cheb_iter}
	\textbf{Input: }{$\bA$, $\bb$, initial guess $\bx_0$, parameters $\alpha,c$ such that $\alpha-c=a$ and $\alpha+c=b$.}\\
	Set $\bb\gets\bb-\bA\bx_0$.\\
	Set $\bp_{0}=\bb$.\\
	Set $S_0=\frac{1}{\sqrt{\alpha-c}\sqrt{\alpha+c}}$.\\
	Set $\bx_{1}=S_0\bp_{0}$.\\
	\For{k=1,2,\ldots}{
		\uIf{k=1}{
			Set $\bp_{1}=\frac{1}{c}(\bA\bp_0-\alpha\bp_0)$.\\
		}
		\Else{
			Set $\bp_{k}=\frac{2}{c}\bA\bp_{k-1}-\frac{2\alpha}{c}\bp_{k-1}-\bp_{k-2}.$\\
		}
		Set $S_k=S_{k-1}\left(-\frac{\alpha}{c}-\sqrt{-1-\frac{\alpha}{c}}\sqrt{1-\frac{\alpha}{c}}\right)$.\\
		Set $\bx_{k+1}=\bx_k+2S_k\bp_{k}.$\\
            \If{\emph{converged}}{
                Set $\bx_{k+1}\gets\bx_{k+1}+\bx_0$.\\
	        \Return{$\bx_{k+1}$
            }
	}
 }
\end{algorithm}
This formula can be implemented as an explicit iteration, see Algorithm \ref{alg:cheb_iter}. We advocate for this modification because, as we describe in this paper, it allows one to directly generalize to general classes of orthogonal polynomials. Furthermore, the asymptotics of such polynomials give error bounds for the iteration. Derivations and comparisons of the classical Chebyshev iteration and our modification are given in Appendix \ref{ap:cheb_iter}.

The assumption that $\bA$ is symmetric and definite can be relaxed provided that the eigenvalues of $\bA$ lie sufficiently close to $[a,b]$ \cite{Manteuffel1977, manteuffel_adaptive_1978}. In fact, given an ellipse containing the spectrum of $\bA$, but not the origin, one can find an optimal interval $[a,b]$ for which the Chebyshev iteration converges (after possible rotation $\bA\to\ex^{\im\phi}\bA$); however, the iteration fails if there is no such ellipse \cite{Manteuffel1977}. 

For instance, assume that the eigenvalues of $\bA$ lie in $[a_1,b_1]\cup[a_2,b_2]$ where $a_1<b_1<0<a_2<b_2$. In this case, to use the iteration \eqref{eq:cheby_iter}, the shifted and scaled Chebyshev polynomials need to be replaced with polynomials with similar properties on multiple intervals to those of the Chebyshev polynomials on a single interval. In the classical view, one needs to replace $r_k$ with polynomials that are small on $[a_1,b_1]\cup[a_2,b_2]$ and large at the origin. de Boor and Rice \cite{de_boor_extremal_1982} explore using the minimax polynomials on $[a_1,b_1]\cup[a_2,b_2]$ for this, and they present a method to compute such polynomials. Unfortunately, such polynomials do not, in general, satisfy a three-term recurrence relation which makes generating the matrix polynomials difficult. Saad \cite{Saad} circumvents this issue by instead using polynomials that are orthogonal with respect to a weight function $w$ supported on $[a_1,b_1]\cup[a_2,b_2]$. Since such polynomials satisfy a three-term recurrence relation, the difficulty lies in computing the recurrence coefficients and the generalization of the integral terms in \eqref{eq:cheby_iter}. Using a clever construction, for this two-interval case, Saad iteratively builds the first $k$ of these quantities in $\OO(k^2)$ arithmetic operations using Chebyshev polynomials defined separately on each interval. Specifically, he computes the degree $k$ polynomial $q_k$ that minimizes a weighted $L^2$ norm $\|1-\diamond q_k(\diamond)\|_{L^2_w}$ for a chosen weight function $w$ and uses their three-term recurrence to compute $q_k(\bA)\bb$ iteratively.  In this work we, in effect, improve upon that in \cite{Saad} by taking $w(x) = x^2 v(x)$ where $v(x)$ is the weight function that Saad employs\footnote{In \cite{Saad}, $v(x)$ is found by using $h_j$ constant, $\alpha_j = \beta_j = -1$ in \eqref{rhp_weight}.}.

In \cite{part1}, we presented a numerical Riemann--Hilbert-based method for computing recurrence coefficients and Cauchy integrals for orthogonal polynomials with weight functions supported on multiple, disjoint intervals (see \cite{Trogdon2013b,TrogdonSORMT} for a method for measures supported on $\mathbb R$). This method computes the first $k$ recurrence coefficients and Cauchy integrals in $\OO(k)$ arithmetic operations. These quantities can be employed analogously to \eqref{eq:cheby_iter}, yielding a generalization of the Chebyshev iteration that improves the complexity of \cite{Saad}. 

In the current work, we discuss how a specific choice of weight function for the orthogonal polynomials allows the numerical method of \cite{part1} to be sped up significantly --- reducing the constant in the $\OO(k)$ bound.  However, this can be further improved in an even more special case. In \cite{akhiezer}, Akhiezer derived explicit formulae for the orthogonal polynomials with respect to a specific Chebyshev-like weight function supported on two disjoint intervals. We review Akhiezer's result and demonstrate how it, along with our numerical method from \cite{part1}, can be used to extend \eqref{eq:cheby_iter} to an iterative method for solving $\bA\bx=\bb$ where $\bA$ has eigenvalues on, or near, multiple, disjoint intervals. We will refer to this method as \emph{the Akhiezer iteration}. Furthermore, we discuss how the Akhiezer iteration can be extended to compute matrix functions $f(\bA)\bb$ for suitably analytic functions $f$, similar to the approach in \cite{SaadMatrixFunction}, and implemented in optimal complexity with a provably exponential convergence rate. In Figure \ref{intro_fig}, we demonstrate the convergence of the Akhiezer iteration applied to a symmetric matrix with spectrum contained in $[-3,-1]\cup[2,4]$.
\begin{figure}
\centering
\begin{subfigure}{0.495\linewidth}
	\centering
	\includegraphics[width=\linewidth]{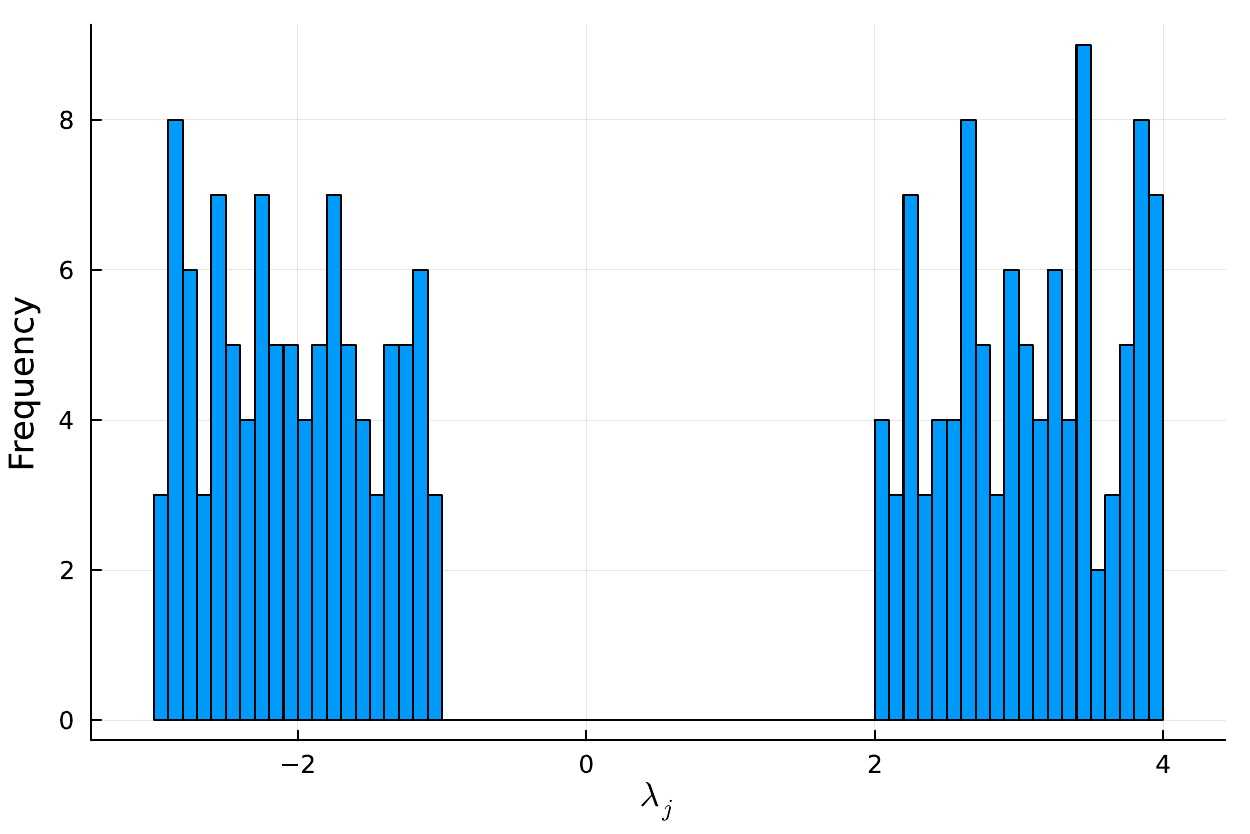}
\end{subfigure}
\begin{subfigure}{0.495\linewidth}
	\centering
	\includegraphics[width=\linewidth]{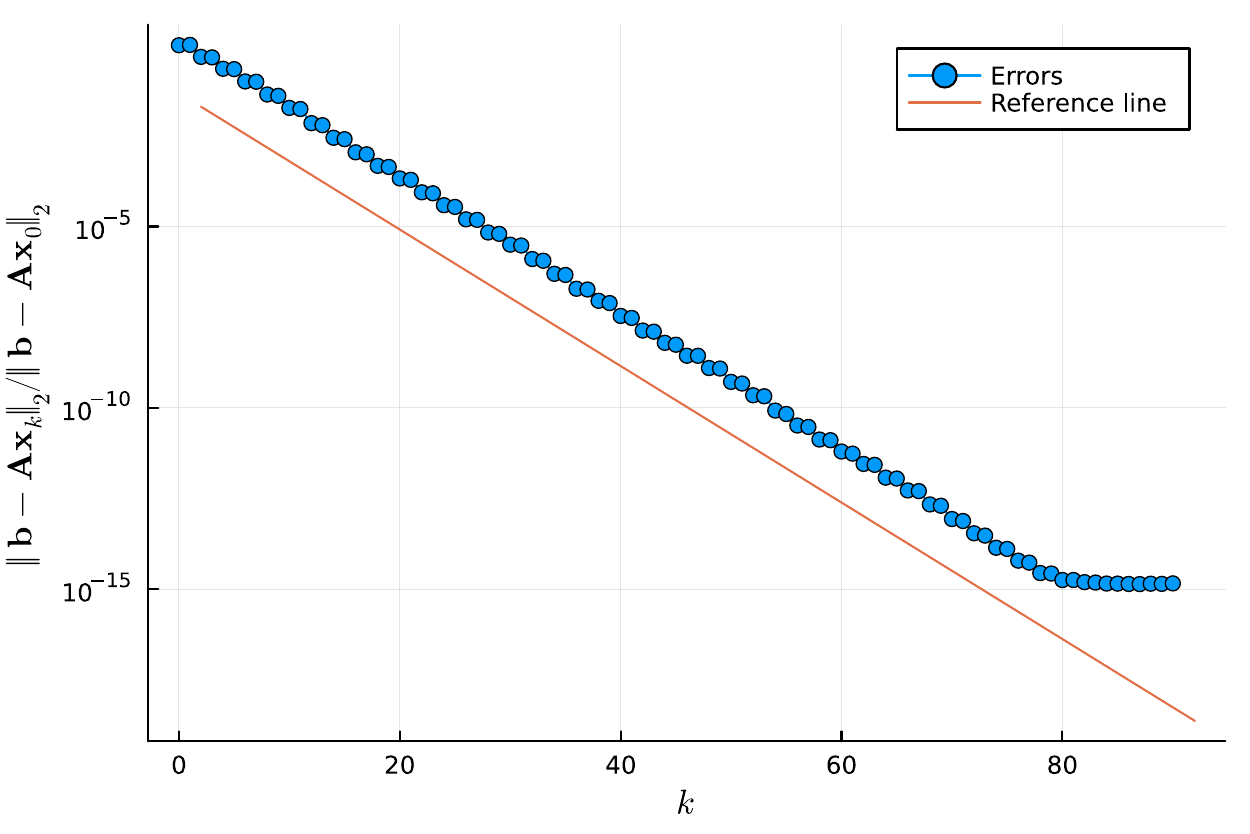}
\end{subfigure}
\caption{Histogram of the eigenvalues of $\bA\in\real^{200\times200}$ (left) and the convergence of the Akhiezer iteration applied to $\bA\bx=\bb$ where $\bb\in\real^{200}$ is a random Gaussian vector. The reference line that approximates the convergence rate is derived in Section \ref{sect:lin_conv}.}
\label{intro_fig}
\end{figure}

The layout of the paper is as follows. In Section \ref{sect:notation}, we establish notation for orthogonal polynomials and their Cauchy integrals. In Section \ref{sect:akh_polys}, we present Akhiezer's orthogonal polynomial construction and discuss how it can be used to computed recurrence coefficients and Cauchy integrals. In Section \ref{sect:akh_iter}, we present the Akhiezer iteration and its extension to matrix functions and discuss convergence rates. In Section \ref{sect:adapt}, we discuss how the Akhiezer iteration can be implemented adaptively when $\bA$ is symmetric. In Section \ref{sect:examples}, we present some examples and applications, and in Section \ref{sect:discuss}, we present our conclusions and discuss potential extensions of our work. Code used to generate the plots in this paper can be found at \cite{code_repo}.

\section{Orthogonal polynomials and their Cauchy integrals}\label{sect:notation}
\subsection{Orthogonal polynomials}
For our purposes, a weight function $w$ is a nonnegative function defined on a finite union of intervals $\Sigma$, $\Sigma\subset\real$, that is continuous and positive on the interior of $\Sigma$ such that $\int_\Sigma w(x)\df x=1$. Consider a sequence of univariate monic polynomials $\pi_0(x),\pi_1(x),\pi_2(x),\ldots$ such that $\pi_j$ has degree $j$ for all $j\in\mathbb{N}$. These polynomials are said to be orthogonal with respect to a weight function $w$ if
\[
\langle\pi_j,\pi_k\rangle_{L^2_w(\Sigma)}=h_j\delta_{jk},
\]
where $h_j>0$, $\delta_{jk}$ is the Kronecker delta, and
\begin{equation*}
	\langle g,h\rangle_{L^2_w(\Sigma)}=\int_\Sigma g(x)\overline{h(x)}w(x)\df x, \quad \|g\|_{L^2_w(\Sigma)}=\sqrt{\langle g,g\rangle_{L^2_w(\Sigma)}}.
\end{equation*}
The orthonormal polynomials $p_0(x),p_1(x),p_2(x),\ldots$ are defined by
\[
p_j(x)=\frac{1}{\sqrt{h_j}}\pi_j(x),
\]
for all $j\in\mathbb{N}$. The orthonormal polynomials satisfy a symmetric three-term recurrence
\begin{equation}\label{recurr}
	\begin{aligned}
		&xp_0(x)=a_0p_0(x)+b_0p_1(x),\\
		&xp_k(x)=b_{k-1}p_{k-1}(x)+a_kp_k(x)+b_kp_{k+1}(x),\quad k\geq1,
	\end{aligned}
\end{equation}
where $b_k>0$ for all $k$. A general reference is \cite{Szego1939}.

\subsection{Cauchy integrals}
Given a curve $\Gamma\subset\mathbb{C}$ and a function $f:\Gamma\to\mathbb{C}$, the Cauchy transform $\mathcal{C}_\Gamma$ is an operator that maps $f$ to its Cauchy integral, i.e.,
\[
\mathcal{C}_\Gamma f(z)=\frac{1}{2\pi \im}\int_\Gamma\frac{f(s)}{s-z}\df s,\quad z\in\compl\setminus\Gamma.
\]
Cauchy integrals and Stieltjes transforms of orthogonal polynomials will also be of use in this work. For orthonormal polynomials $p_k(x)$ corresponding to the weight function $w$ on $\Sigma\subset\real$, define 
\[
C_k(z)=\mathcal{C}_\Sigma[p_kw](z)=\frac{1}{2\pi \im}\int_\Sigma\frac{p_k(s)w(s)}{s-z}\df s, \quad \mathcal{S}_\Sigma[p_kw](z)=2\pi \im\mathcal{C}_\Sigma[p_kw](z),
\] 
for $z\in\mathbb{C}\setminus \Sigma$.
It is well-known that if the three-term recurrence for $p_k$ is given by \eqref{recurr}, these Cauchy integrals also satisfy the three-term recurrence \cite{olver_slevinsky_townsend_2020}:
\begin{equation}\label{intrecurr}
	\begin{aligned}
		&zC_0(z)=a_0C_0(z)+b_0C_1(z)-\frac{1}{2\pi \im},\\
		&zC_k(z)=b_{k-1}C_{k-1}(z)+a_kC_k(z)+b_kC_{k+1}(z),\quad k\geq1.
	\end{aligned}
\end{equation}

\begin{remark}
    It is important to note that the direct numerical use of \eqref{intrecurr} is hampered by the fact that it is intrinsically unstable.  Cauchy integrals decay exponentially in the complex plane, $z \not \in \Sigma$, whereas the orthogonal polynomials grow exponentially. Any error, say one on the order of machine precision, will be amplified exponentially as the recurrence is run.
\end{remark}

\section{Akhiezer polynomials}\label{sect:akh_polys}
In \cite{part1}, we described a Riemann--Hilbert-based numerical method for computing orthogonal polynomials corresponding to genus $g$ weight functions of the form \begin{equation}\label{rhp_weight}w(x)=\sum_{j=1}^{g+1}\mathbbm{1}_{[a_j,b_j]}(x)h_j(x)\left(\sqrt{x-a_j}\right)^{\alpha_j}\left(\sqrt{b_j-x}\right)^{\beta_j},\quad \alpha_j,\beta_j\in\{-1,1\},
\end{equation}
where $[a_j,b_j]$, $j=1,\ldots,g+1$ are disjoint and $h_j$ is positive on $[a_j,b_j]$ and has an analytic extension to a neighborhood $\Omega_j$ of $[a_j,b_j]$. A judicious choice of $h_j$ yields the weight function 
\begin{equation}\label{special_weight}
	w(x)\propto\mathbbm{1}_{\Sigma}(x)\frac{\sqrt{b_{g+1}-x}\prod_{j=1}^{g+1}\sqrt{x-a_j}}{\prod_{j=1}^{g}\sqrt{x-b_j}}.
\end{equation}
If the weight function is restricted to be of the form \eqref{special_weight}, the underlying Riemann--Hilbert problem can be simplified greatly, and the numerical method can be sped up significantly. See Appendix \ref{ap:rhp} for a full discussion of this. In a sense, weight functions of the form \eqref{special_weight} are a natural extension of the Chebyshev polynomials, as they limit to a scaled and shifted Chebyshev-second-kind weight as the gaps between the intervals approach zero.  




The Akhiezer polynomials are defined as the orthogonal polynomials with respect to the weight function
\begin{equation}\label{eq:gen_akh_weight}
w(x)=\frac{1}{\pi}\mathbbm{1}_{\Sigma}(x)\frac{\prod_{j=1}^{g}\sqrt{x-b_j}}{\sqrt{b_{g+1}-x}\prod_{j=1}^{g+1}\sqrt{x-a_j}},
\end{equation}
which is the normalized reciprocal of \eqref{special_weight} and corresponds to a Chebyshev-first-kind-like weight \cite{Chen2007}. Formulae for such polynomials can be worked out explicitly in terms of Riemann theta functions \cite{Chen2002}. In the case $g=1$, the Riemann theta functions reduce to Jacobi theta functions, and the polynomials can be computed efficiently via these theta functions. In the following, we summarize Akhiezer's construction of the polynomials from \cite[Section 53]{akhiezer} and derive formulae for their Cauchy integrals and recurrence coefficients.

\subsection{Akhiezer's formulae and notation}
Following Akhiezer's convention, we consider\footnote{This can be generalized to any configuration of two disjoint intervals by shifting and scaling the resulting polynomials as appropriate.} $\Sigma=[-1,\alpha]\cup[\beta,1]$ such that $-1<\alpha<\beta<1$ and define the weight function 
\begin{equation}\label{akh_weight}
	w(x)=\frac{1}{\pi}\mathbbm{1}_\Sigma(x)\frac{\sqrt{x-\alpha}}{\sqrt{1-x}\sqrt{x+1}\sqrt{x-\beta}}.
\end{equation}
Then, the orthonormal polynomials with respect to the weight function $w$ can be computed according to the following procedure: First, compute the elliptic modulus\footnote{In this subsection, we follow standard notation and let $k$ be the elliptic modulus, not the iteration number of an algorithm.} \cite[Equation 6]{akhiezer}
\begin{equation*}
	k=\sqrt{\frac{2(\beta-\alpha)}{(1-\alpha)(1+\beta)}},
\end{equation*}
which will serve as the elliptic modulus in all Jacobi elliptic functions that follow. Additionally, $K=K(k)$ is the complete elliptic integral:
\begin{equation*}
	K(k)=\int_{0}^{1}\frac{\df t}{\sqrt{(1-t^2)(1-k^2t^2)}}=\int_{0}^{\pi/2}\frac{\df\theta}{\sqrt{1-k^2\sin^2\theta}}.
\end{equation*}
Then, the parameter $\rho$ is defined so that 
\begin{equation*}
	1-2\sn^2\rho=\alpha,
\end{equation*}
and $0<\rho<K$ where $\sn(\diamond)=\sn(\diamond,k)$ denotes the Jacobi $\sn$ function \cite[Equation 10]{akhiezer}. Since Jacobi elliptic functions are defined as the inverses of elliptic integrals, $\rho$ can be written as 
\begin{equation*}
	\rho=F\left(\arcsin\sqrt{\frac{1-\alpha}{2}},k\right),
\end{equation*}
where $F$ is the incomplete elliptic integral of the first kind defined by
\begin{equation*}
	F(\phi,k)=\int_0^{\sin\phi}\frac{\df t}{\sqrt{(1-t^2)(1-k^2t^2)}}=\int_{0}^{\phi}\frac{\df\theta}{\sqrt{1-k^2\sin^2\theta}}.
\end{equation*}
Now, given $x\in\compl$, define $u=u(x)$ such that \cite[Equation 9]{akhiezer}
\begin{equation*}
	x-\alpha=\frac{1-\alpha^2}{2\sn^2(u)+\alpha-1}.
\end{equation*}
An explicit formula for $u$ is
\begin{equation}\label{u_formula}
	u=F\left(\arcsin\sqrt\frac{(\alpha-1)(1+x)}{2(\alpha-x)},k\right).
\end{equation}
Define the Jacobi theta functions
\begin{align*}
	&\theta_1(z,q)=2\sum_{j=0}^{\infty}(-1)^jq^{\left(j+\frac{1}{2}\right)^2}\sin\left((2j+1)z\right),\\
	&\theta_4(z,q)=1+2\sum_{j=1}^{\infty}(-1)^jq^{j^2}\cos\left(2jz\right).
\end{align*}
In what follows, we will always take the parameter $q$ to be given by
\begin{equation*}
	q=q(k)=\exp\left(-\pi \frac{K(\sqrt{1-k^2})}{K(k)}\right).
\end{equation*}
Now, define the functions
\begin{equation*}
	H(z)=\theta_1\left(\frac{\pi z}{2K},q\right),\quad \Theta(z)=\theta_4\left(\frac{\pi z}{2K},q\right).
\end{equation*}
Then, the $n$th degree orthonormal polynomial $p_n$ is given by \cite{akhiezer}
\begin{equation}\label{poly}
	p_n(x)=\frac{C_n}{2}\left(\left(\frac{H(u-\rho)}{H(u+\rho)}\right)^n\frac{\Theta(u+2n\rho)}{\Theta(u)}+\left(\frac{H(u+\rho)}{H(u-\rho)}\right)^n\frac{\Theta(u-2n\rho)}{\Theta(u)}\right),
\end{equation}
where
\begin{equation*}
	C_n=\frac{\sqrt{2}\Theta(\rho)}{\sqrt{\Theta\left((2n-1)\rho\right)\Theta\left((2n+1)\rho\right)}},
\end{equation*}
for $n\geq1$, $C_0=1$, and $u$ is given by \eqref{u_formula}.

Akhiezer's derivation also includes a formula for the Cauchy integrals of these polynomials:
\begin{equation}\label{eq:cauchy_ints}
	\mathcal{C}_\Sigma\left[p_nw\right](x)=-\frac{C_n}{2\pi\im}\left(\frac{H(u-\rho)}{H(u+\rho)}\right)^n\frac{\Theta(u+2n\rho)}{\Theta(u)}\frac{\sqrt{x-\alpha}}{\sqrt{x-1}\sqrt{x+1}\sqrt{x-\beta}},
\end{equation}
with $u$ and $C_n$ defined as above.

In Appendix \ref{ap:recurr}, we derive the following formulae for the recurrence coefficients of the Akhiezer polynomials:
\begin{equation}\label{eq:akh_recurr}
\begin{aligned}
	a_0=&(1-\alpha^2)\left(\frac{1}{8\sn^2\rho}+\frac{\sn''\rho}{8\sn\rho\left(\sn'\rho\right)^2}+\frac{H'(2\rho)}{H(2\rho)}\frac{1}{4\sn\rho\sn'\rho}\right)+\frac{1-\alpha^2}{2\sn\rho\sn'\rho}\frac{\Theta'\left(\rho\right)}{\Theta\left(\rho\right)}+\alpha,\\
	a_n=&(1-\alpha^2)\left(\frac{1}{8\sn^2\rho}+\frac{\sn''\rho}{8\sn\rho\left(\sn'\rho\right)^2}\right)\\&+\frac{1-\alpha^2}{4\sn\rho\sn'\rho}\left(\frac{H'(2\rho)}{H(2\rho)}+\frac{\Theta'\left((2n+1)\rho\right)}{\Theta\left((2n+1)\rho\right)}-\frac{\Theta'\left((2n-1)\rho\right)}{\Theta\left((2n-1)\rho\right)}\right)+\alpha,\quad n\geq1,\\
	b_0=&\frac{1-\alpha^2}{4\sn\rho\sn'\rho}\frac{H'(0)}{H(2\rho)}\sqrt{\frac{2\Theta(3\rho)}{\Theta(\rho)}},\\
	b_n=&\frac{1-\alpha^2}{4\sn\rho\sn'\rho}\frac{H'(0)}{H(2\rho)}\frac{\sqrt{\Theta((2n+3)\rho)\Theta((2n-1)\rho)}}{\Theta((2n+1)\rho)},\quad n\geq1.
\end{aligned}
\end{equation}

Both the polynomials themselves and their Cauchy integrals can be efficiently computed via standard algorithms for computing Jacobi theta functions and elliptic integrals \cite[Sections 19.36, 20.14]{NIST:DLMF} while the recurrence coefficients can be efficiently computed via standard algorithms for computing Jacobi elliptic \cite[Section 22.20]{NIST:DLMF} and Jacobi theta functions. In particular, the parameter $q$ is typically very small and grows very slowly as the size of the intervals shrinks, so the Jacobi theta functions can be computed by simply summing the first few terms of their respective series. However, one weakness of these formulae is that the mapped variable $u(x)$ is undefined at $x=\alpha$. While, by continuity, $u(\alpha)=\im K(\sqrt{1-k^2})$, the non-Lipschitz behavior of the involved functions leads to numerical instabilities when this value is used in the formulae. Thus, if one needs to compute the polynomials at this endpoint, it is better to generate them via the recurrence coefficients rather than using \eqref{poly}.

Otherwise, we find that these formulae are extremely efficient for evaluating the Akhiezer polynomials, their Cauchy integrals, and their recurrence coefficients in this special $g=1$ case, outperforming by multiple orders of magnitude both our $\OO(n)$ Riemann--Hilbert approach and the optimized $\OO(n^2)$ algorithm given as RKPW in \cite{Gragg1984} and \texttt{lanczos.m} in \cite{gautschi}. Thus, these formulae should be used for all two interval problems for which one can choose the weight function to be of the form \eqref{akh_weight}. For higher genus problems, choosing a weight function of the form \eqref{special_weight} allows our Riemann--Hilbert approach to be applied efficiently at all degrees. If one wishes to consider a more general weight function of the form \eqref{rhp_weight}, they should consider applying the RKPW algorithm for the low degree polynomials then switching to the asymptotic version of our Riemann--Hilbert approach for higher degrees. The Cauchy integrals for the lower degree polynomials can then be obtained by applying the three-term recurrence \eqref{intrecurr} to those obtained from our Riemann--Hilbert approach. Indeed, \eqref{intrecurr} can be written as the linear system
\begin{equation*}
	\begin{pmatrix}
		a_0-z &b_0\\
		b_0 &a_1-z &b_1\\
		&b_1 &a_2-z &b_2\\
		&&\ddots &\ddots &\ddots
	\end{pmatrix}\begin{pmatrix}
		C_0(z)\\C_1(z)\\C_2(z)\\\vdots
	\end{pmatrix}=\begin{pmatrix}
		\frac{1}{2\pi \im}\\0\\0\\\vdots
	\end{pmatrix}.
\end{equation*}
If for some $N$, $C_{N+2}(z)$ and $C_{N+1}(z)$ are known, then the truncated linear system linear system 
\begin{equation*}
	\begin{pmatrix}
		a_0-z &b_0\\
		b_0 &a_1-z &b_1\\
		&\ddots &\ddots &\ddots\\
		&&b_{N-2}&a_{N-1}-z&b_{N-1}\\
		&&&b_{N-1}&a_N-z\\
		&&&&b_N
	\end{pmatrix}\begin{pmatrix}
		C_0(z)\\C_1(z)\\\vdots\\C_{N-1}(z)\\C_{N-2}(z)\\C_N(z)
	\end{pmatrix}=\begin{pmatrix}
		\frac{1}{2\pi \im}\\0\\\vdots\\0\\
		-b_NC_{N+1}(z)\\
		(z-a_{N+1})C_{N+1}(z)-b_{N+1}C_{N+2}(z)
	\end{pmatrix},
\end{equation*}
can be solved efficiently to obtain $C_0(z),\ldots,C_N(z)$. As noted previously, the forward recurrence \eqref{intrecurr} becomes more numerically unstable the further $z$ is from $\Sigma$, so it is not, in general, a reliable way to compute Cauchy integrals.

\section{The Akhiezer iteration and matrix function evaluation}\label{sect:akh_iter}
\subsection{The Akhiezer iteration for solving linear systems}
Following \cite{Saad}, orthogonal polynomials with respect to a weight function $w$ of the form \eqref{rhp_weight} (in particular \eqref{akh_weight}) can be used to extend the classical Chebyshev iteration to solve indefinite linear systems. Assume that a square matrix $\bA$ has eigenvalues on, or near, $\Sigma=\bigcup_{j=1}^{g+1}[a_j,b_j]$ and that $0\notin\Sigma$. If $p_0,p_1,\ldots$ denote the orthonormal polynomials with respect to $w$, then a $p_j$-series expansion for $1/(x-z)$, for $x \in \Sigma$,  $z\notin\Sigma$, is given by
\begin{equation}\label{eq:gen_expand}
	\begin{aligned}
\frac{1}{x-z}&=\sum_{j=0}^{\infty}\left\langle p_j,\frac{1}{\diamond -z}\right\rangle p_j(x)=\sum_{j=0}^{\infty}\left(\int_\Sigma\frac{p_j(s)}{s-z}w(s)\df s \right)p_j(x)=
\sum_{j=0}^{\infty}\mathcal{S}_\Sigma\left[p_jw\right](z)p_j(x).
	\end{aligned}
\end{equation}
This can be extended to an iterative method for solving the linear system $\bA \bx=\bb$ by truncating the series with $z = 0$:
\begin{equation}\label{eq:iteration}
	\bx=\bA^{-1}\bb=\sum_{j=0}^{\infty}\mathcal{S}_\Sigma\left[p_jw\right](0)p_j(\bA)\bb\approx\sum_{j=0}^{k}\mathcal{S}_\Sigma\left[p_jw\right](0)p_j(\bA)\bb=:\bx_{k+1}.
\end{equation}
We call the iterative implementation of this approximation \emph{the Akhiezer iteration}. The recurrence coefficients and Cauchy integrals can be computed either by our Riemann--Hilbert approach \cite{part1} or by \eqref{eq:cauchy_ints} and \eqref{eq:akh_recurr} in the case of the $g=1$ Akhiezer weight. Then, the Akhiezer iteration for shifted linear solves can be implemented as in Algorithm \ref{alg:akh_iter}.
\begin{algorithm}
	\caption{Akhiezer iteration to approximate $(\bA - z \bI)^{-1} \bb$, $z \not \in \Sigma$}\label{alg:akh_iter}
	\textbf{Input: }{$\bA$, $\bb$, initial guess $\bx_0$, and a function to compute recurrence coefficients $a_k,b_k$ and Stieltjes transforms $S_k=\mathcal{S}_\Sigma\left[p_kw\right](z)$.}\\
	Set $\bb\leftarrow\bb-\bA\bx_0$.\\
	Set $\bp_{0}=\bb$.\\
	Set $\bx_{1}=S_0\bp_{0}$.\\
	\For{k=1,2,\ldots}{
		\uIf{k=1}{
			Set $\bp_{1}=\frac{1}{b_0}(\bA\bp_0-a_0\bp_0)$.
		}
		\Else{
			Set $\bp_{k}=\frac{1}{b_{k-1}}(\bA\bp_{k-1}-a_{k-1}\bp_{k-1}-b_{k-2}\bp_{k-2}).$
		}
		Set $\bx_{k+1}=\bx_k+S_k\bp_{k}.$\\
            \If{\emph{converged}}{
            Set $\bx_{k+1}\leftarrow\bx_{k+1}+\bx_0$.\\
            \Return{$\bx_{k+1}$
            }
	}
	}
\end{algorithm}

Algorithm \ref{alg:akh_iter} can instead be utilized to build an approximation to the matrix inverse $\bA^{-1}$ by using $\bb=\bI$ and removing the initial guess $\bx_0$. Then, at each step, the algorithm requires matrix-matrix multiplication rather than matrix-vector products but is otherwise identical.

\begin{remark}\label{r:diff}
    In \eqref{eq:gen_expand}, one can differentiate with respect to $z$ a number of times to find expansions for $x^{-p}$ for $p \in \mathbb N$, provided one can compute derivatives of the Stieltjes transforms $\mathcal S_{\Sigma}$. The Riemann--Hilbert approach \cite{part1} allows for this computation in a fairly straightforward manner, but we do not implement this here. 
\end{remark}

\subsubsection{Convergence}\label{sect:lin_conv}
From \cite{Ding2022}, we have the following behavior of the orthogonal polynomials:
\begin{lemma}\label{poly_asymp}
	Fix $\epsilon>0$ and the set $V=\{z\in\compl~:~|z-a_j|\geq\epsilon,~|z-b_j|\geq\epsilon,~j=1,\ldots,g+1\}$. Then,
	\begin{equation*}
		\begin{aligned}
		&p_n(z)=\hat\delta_n(z)\ex^{n\mathfrak{g}(z)},&\quad&2\pi\im\mathcal{C}_\Sigma\left[p_nw\right](z)=\delta_n(z)\ex^{-n\mathfrak{g}(z)},\quad z\in V,\\
		\end{aligned}
	\end{equation*}
where $\delta_n(z)$ and $\hat\delta_n(z)$ are uniformly bounded in both $n$ and $z\in V$ and $\mathfrak g$ is the exterior Green's function with pole at infinity\footnote{See Appendix~\ref{ap:rhp} for the construction of $\mathfrak g$.  We point out that it is determined by $\Sigma$ alone.} corresponding to $\Sigma$.  Furthermore, $\ex^{\mathfrak g(z)}$ is a conformal map from $\mathbb C \setminus \Sigma$ to the exterior of the unit disk satisfying $\re \mathfrak g(z) = 0$ for $z \in \Sigma$.
\end{lemma}
\begin{proof}
By \cite[Equations B.2, B.3]{Ding2022} and undoing the lensing step of the Riemann--Hilbert problem deformations, the above formulae hold for $\delta_n(z)$ and $\hat\delta_n(z)$ depending on bounded quantities and $\prod_{j=0}^{n-1}\frac{1}{b_j\mathfrak{c}}$ where $\log \mathfrak c$ is the $O(1)$ term in the expansion of $\mathfrak g$ at infinity. By \cite[Equations B.3, B.12]{Ding2022}, 
\begin{equation*}
\mathfrak{c}^{n}\prod_{j=0}^{n-1}b_j=\OO(1),
\end{equation*}
as $n\to\infty$, so $\delta_n(z)$ and $\hat\delta_n(z)$ are bounded from above.  The rest of the claims follow immediately from \cite{Ding2022}.
\end{proof}
In the $g=1$ case,
\begin{equation*}
\ex^{\mathfrak g(z)}=-\frac{H(u+\rho)}{H(u-\rho)},
\end{equation*} 
in the notation of Akhiezer's formulation \eqref{poly}. When $g$ is larger than one, $\mathfrak g(z)$ can be computed numerically as described in \cite[Section 4.2]{part1}. 

To ease presentation, we introduce the following definitions.
\begin{definition}
We say that a matrix is generic if it is diagonalizable and has no eigenvalues at the endpoints of $\Sigma$.
\end{definition}
This definition is here to simplify the analysis, and our theorems require the matrix in question to be generic.  With a more technical approach, one can obtain similar bounds for non-generic matrices and the algorithms in question still converge at roughly the same rate. We remark on this below.

\begin{definition}
Given $\Sigma$ and $\bA \in \mathbb C^{n \times n}$ define
\begin{align*}
    \nu(z;\bA)=\max_j\re\mathfrak g(\lambda_j)-\re \mathfrak{g}(z), \quad \nu(\bA)=\max_j\re\mathfrak g(\lambda_j).
\end{align*}
where $\lambda_j$ are the eigenvalues of $\bA$.
\end{definition}

In the following, $\|\diamond\|_2$ will denote the Euclidean 2-norm and the induced matrix norm.
\begin{theorem}\label{thm:lin_error}
Consider a density $w$ of the form \eqref{rhp_weight}.  If $\bA\in\compl^{n\times n}$ is generic and its eigenvalues $\lambda_j$ all satisfy $\re\mathfrak g(\lambda_j)<\re \mathfrak{g}(z)$, then there exists $c> 0$, depending only on the eigenvalues of $\bA$, such that
\begin{equation*}
E_k(z)=\left\|
\sum_{j=0}^{k-1}\mathcal{S}_\Sigma\left[p_jw\right](z)p_j(\bA)\bb-(\bA-z\bI)^{-1}\bb\right\|_2,
\end{equation*}
satisfies
\begin{equation*}
E_k(z)\leq c\|\bV\|_2\left\|\bV^{-1}\right\|_2\|\bb\|_2\frac{\ex^{k\nu(z;\bA)}}{1-\ex^{\nu(z;\bA)}},
\end{equation*}
 where $\bA$ is diagonalized as $\bA=\bV\bLambda\bV^{-1}$, $\bLambda = {\rm diag}(\lambda_1,\ldots,\lambda_n)$. In particular, the Akhiezer iteration converges, at worst, at a geometric rate given by $\ex^{\nu(z;\bA)}$ if $z\notin\Sigma$ since
 \begin{align*}
     \| \bx_k - (\bA - z \bI)^{-1}\bb \|_2 = E_k(z).
 \end{align*}
\end{theorem}
\begin{proof}
We have
\begin{equation*}
\begin{aligned}
&E_k(z)=\left\|
\sum_{j=0}^{k-1}\mathcal{S}_\Sigma\left[p_jw\right](z)p_j(\bA)\bb-(\bA-z\bI)^{-1}\bb\right\|_2=\left\|\sum_{j=k}^{\infty}\mathcal{S}_\Sigma\left[p_jw\right](z)p_j(\bA)\bb\right\|_2\\&=
\left\|\sum_{j=k}^{\infty}2\pi\im\mathcal{C}_\Sigma\left[p_jw\right](z)\bV p_j(\bLambda)\bV^{-1}\bb\right\|_2
\\ &\leq
\|\bV\|_2\left\|\bV^{-1}\right\|_2\|\bb\|_2\sum_{j=k}^{\infty}\max_{\ell\in\{1,\ldots,n\}}\left|2\pi\im\mathcal{C}_\Sigma\left[p_jw\right](z)p_j(\lambda_\ell)\right|.
\end{aligned}
\end{equation*}
Since $\lambda_\ell$ is not at an endpoint of $\Sigma$, by Lemma \ref{poly_asymp},
\begin{equation*}
\left|2\pi\im\mathcal{C}_\Sigma\left[p_jw\right](z)p_j(\lambda_\ell)\right|=|\delta_j(z)||\hat\delta_j(\lambda_\ell)|\ex^{j\re\left(\mathfrak g(\lambda_\ell)-\mathfrak{g}(z)\right)}\leq c\ex^{j\nu(z;\bA)},
\end{equation*}
for some constant $c$. The theorem follows from
\begin{equation*}
\begin{aligned}
\sum_{j=k}^{\infty}\ex^{j\nu(z;\bA)}= \frac{\ex^{k\nu(z;\bA)}}{1-\ex^{\nu(z;\bA)}}.
\end{aligned}
\end{equation*}
\end{proof}

\begin{remark}
Theorem \ref{thm:lin_error} could, in principle, be extended to a broader class of matrices. For instance, if $\bA$ does have eigenvalues at the endpoints of $\Sigma$, the orthogonal polynomials at these endpoints are either bounded or grow at a polynomial rate. Since the convergence rate is geometric, this growth could be accounted for in the error bound without inhibiting convergence. However, one could instead perturb $\Sigma$ so that these eigenvalues lie in the interior, and the iteration will suffer a small perturbation to its geometric rate of convergence.
\end{remark}
\begin{remark}
    Jordan blocks, i.e., degenerate eigenspaces for $\bA$, could be accounted for by investigating the derivatives of the orthogonal polynomials.  Due to the Markov brothers' inequality, \cite[Theorem 1.10]{rivlin_81}, it is clear that this will not change the exponential rate of convergence.  Yet, more refined estimates are possible.
\end{remark}

\subsection{The Akhiezer iteration for matrix function approximation}

The Akhiezer iteration can be extended to compute matrix functions for square matrices $\bA$ with eigenvalues on or near $\Sigma=\bigcup_{j=1}^{g+1}[a_j,b_j]$. Consider the Cauchy integral representation of a matrix function
\begin{equation*}
f(\bA)\bb=\frac{1}{2\pi\im}\int_\Gamma f(z)(z\bI-\bA)^{-1}\bb\df z,
\end{equation*}
where $\Gamma$ is a counterclockwise oriented curve that encloses the spectrum of $\bA$ and $f$ is analytic in a region containing $\Gamma$. Applying a quadrature rule, such as a trapezoid rule \cite{trapezoid}, with nodes $z_j$ and weights $w_j$, this can be approximated as
\begin{equation*}
	f(\bA)\bb\approx\frac{1}{2\pi\im}\sum_{j=1}^m f(z_j)w_j(z_j\bI-\bA)^{-1}\bb.
\end{equation*}
In practice, we find that it typically suffices to take $\Gamma=\bigcup_{j=1}^{g+1}\Gamma_j$ to be a collection of circles around each $[a_j,b_j]$. In this case the trapezoid rule on each $\Gamma_k$ has weights and nodes
\begin{equation*}
w_j=\frac{2\pi\im r_k}{m_k}\ex^{\im\theta_j^{(k)}},\quad z_j=r_k\ex^{\im\theta_j^{(k)}}+c_k,\quad j=1,\ldots,m_k,
\end{equation*}
where $\{\theta_j^{(k)}\}_{j=1}^{m_k}$ are evenly spaced points in $[0,2\pi)$, $r_k$ and $c_k$ are the radius and center of $\Gamma_k$, respectively, and $m_k$ is the number of quadrature nodes on $\Gamma_k$, $\sum_{j=1}^{g+1}m_j=m$.  For convenience, we reorder the weights and nodes for all contours $\Gamma_k$ as
\begin{align*}
    z_j, w_j, \quad j = 1,2,\ldots,m,
\end{align*}
with the first $m_1$ corresponding to $\Gamma_1$, the next $m_2$ corresponding to $\Gamma_2$, and so on.

To approximate $f(\bA)\bb$, the resolvent is approximated by
\begin{equation*}
(z_j\bI-\bA)^{-1}=-\sum_{\ell=0}^{\infty}\mathcal{S}_\Sigma\left[p_\ell w\right](z_j)p_\ell(\bA)\approx-\sum_{\ell=0}^{k-1}\mathcal{S}_\Sigma\left[p_\ell w\right](z_j)p_\ell(\bA),
\end{equation*}
and the matrix function is approximated by
\begin{equation}\label{mat_func_iter}
f(\bA)\bb\approx \bff_k := f_{k,m}(\bA)\bb:=-\sum_{\ell=0}^{k-1}\left( \sum_{j=1}^mf(z_j)w_j\mathcal{C}_\Sigma\left[p_\ell w\right](z_j)\right)p_\ell(\bA)\bb.
\end{equation}
As before, the recurrence coefficients and Cauchy integrals can be precomputed either by our Riemann--Hilbert approach \cite{part1} or by \eqref{eq:cauchy_ints} and \eqref{eq:akh_recurr} in the case of the $g=1$ Akhiezer weight. Then, the Akhiezer iteration for matrix function approximation can be implemented as in Algorithm~\ref{alg:akh_func}.
\begin{algorithm}
	\caption{Akhiezer iteration for matrix function approximation}\label{alg:akh_func}
	\textbf{Input: }{$f$, $\bA$, $\bb$, quadrature nodes $z_1,\ldots,z_m$, quadrature weights $w_1,\ldots,w_m$, and a function to compute recurrence coefficients $a_k,b_k$ and Cauchy integrals $C_{k,j}=\mathcal{C}_\Sigma\left[p_k w\right](z_j)$.}\\
	Set $\bff_0=0$.\\
	\For{k=0,1,\ldots}{
		\uIf{k=0}{
			Set $\bp_0=\bb$.
		}
		\uElseIf{k=1}{
			Set $\bp_{1}=\frac{1}{b_0}(\bA\bp_0-a_0\bp_0)$.
		}
		\Else{
			Set $\bp_{k}=\frac{1}{b_{k-1}}(\bA\bp_{k-1} - a_{k-1}\bp_{k-1} -  b_{k-2}\bp_{k-2}).$
		}
            Set $\alpha_k = -\sum_{j=1}^m C_{k,j} f(z_j) w_j$.\\
            Set $\bff_{k+1}=\bff_k + \alpha_k \bp_k$.\\
            \If{\emph{converged}}{\Return{$\bff_{k+1}$}.}
	}
\end{algorithm}
\begin{remark}
    The approximation to $f$ need not be from the discretization of a Cauchy integral.  It could be from a Pad\'e approximation, a AAA approximation \cite{Nakatsukasa2018}, or some other form of rational approximation.  As long as the approximation can be written in pole-residue form, the Akhiezer iteration can be applied. We refer the reader back to Remark~\ref{r:diff} for the case when the approximation contains more than just simple poles. 
\end{remark}

%
%

\subsubsection{Convergence}
\begin{theorem}\label{thm:func_error}
	Suppose $\bA\in\compl^{n\times n}$ is generic, its eigenvalues $\lambda_\ell$ satisfy $\re\mathfrak g(\lambda_\ell)<\min_{j=1,\ldots,m}\re \mathfrak{g}(z_j)$, $|f(z)|\leq M$ for $z\in\Gamma$, and $\sum_{j=1}^{m}|w_j|\leq 2\pi L$. Then,
	\begin{equation*}
	\left\|f(\bA)\bb-f_{k,m}(\bA)\bb\right\|_2\leq \frac{1}{2\pi}\|\bV\|_2\left\|\bV^{-1}\right\|_2\|\bb\|_2\max_{\ell\in\{1,\ldots,n\}}|\rho_m(\lambda_\ell)|+LM\max_{j\in\{1,\ldots,m\}}E_k(z_j),
	\end{equation*}
	where $\bA$ is diagonalized as $\bA=\bV\bLambda\bV^{-1}$, $\bLambda = {\rm diag}(\lambda_1,\ldots,\lambda_n)$, and
 \begin{equation*}
\rho_m(x)=\int_\Gamma \frac{f(z)}{z-x}\df z-\sum_{j=1}^m \frac{f(z_j)}{z_j-x}w_j.
\end{equation*}
\end{theorem}
\begin{proof}
By the triangle inequality, 
\begin{equation*}
\begin{aligned}
&\left\|f(\bA)\bb-f_{k,m}(\bA)\bb\right\|_2\leq\left\|f(\bA)\bb-\frac{1}{2\pi\im}\sum_{j=1}^m f(z_j)w_j(z_j\bI-\bA)^{-1}\bb\right\|_2\\&+\left\|\frac{1}{2\pi\im}\sum_{j=1}^m f(z_j)w_j(z_j\bI-\bA)^{-1}\bb-f_{k,m}(\bA)\bb\right\|_2.
\end{aligned}
\end{equation*}
Then,
\begin{equation*}
\begin{aligned}
&\left\|f(\bA)\bb-\frac{1}{2\pi\im}\sum_{j=1}^m f(z_j)w_j(z_j\bI-\bA)^{-1}\bb\right\|_2\\&=
\left\|\frac{1}{2\pi\im}\bV^{-1}\int_\Gamma f(z)(z\bI-\bLambda)^{-1}\bV\bb\df z-\frac{1}{2\pi\im}\bV^{-1}\sum_{j=1}^m f(z_j)w_j(z_j\bI-\bLambda)^{-1}\bV \bb \right\|_2\\&\leq\frac{1}{2\pi}\|\bV\|_2\left\|\bV^{-1}\right\|_2\|\bb\|_2\max_{\ell\in\{1,\ldots,n\}}|\rho_m(\lambda_\ell)|.
\end{aligned}
\end{equation*}
Then, in the notation of Theorem \ref{thm:lin_error},
\begin{equation*}
\begin{aligned}
&\left\|\frac{1}{2\pi\im}\sum_{j=1}^m f(z_j)w_j(z_j\bI-\bA)^{-1}\bb-f_{k,m}(\bA)\bb\right\|_2\\&\leq \frac{1}{2\pi}\sum_{j=1}^m\left\| f(z_j)w_j(z_j\bI-\bA)^{-1}\bb+f(z_j)w_j\sum_{\ell=0}^{k-1}\mathcal{S}_\Sigma\left[p_\ell w\right](z_j)p_\ell(\bA)\bb \right\|_2\\&=
\frac{1}{2\pi}\sum_{j=1}^m |f(z_j)||w_j|E_k(z_j)\leq LM\max_{j\in\{1,\ldots,m\}}E_k(z_j),
\end{aligned}
\end{equation*}
and the result follows.
\end{proof}
For a reasonable quadrature rule, $\lim_{m\to\infty}\rho_m(x)=0$ for all $x$, likely exponentially, so provided that $m$ is taken to be sufficiently large, the Akhiezer iteration converges, at worst, at an asymptotic  geometric rate given by $\ex^{\max_j\nu(z_j;\bA)}$ provided that $\bA$ is generic and its eigenvalues $\lambda_\ell$ all satisfy $\re\mathfrak g(\lambda_\ell)<\min_{j=1,\ldots,m}\re \mathfrak{g}(z_j)$. In practice, we find that this bound is quite pessimistic, so we offer an alternative view. We first state a lemma that generalizes estimates for the Chebyshev series for functions analytic in a Bernstein ellipse.
\begin{lemma}\label{bernstein}
Let $\Gamma_\varrho$ denote the level curve $\ex^{\re \mathfrak g(z)}=\varrho$ for $\varrho>1$, and assume that $f$ is analytic in the interior of $\Gamma_\varrho$ and satisfies $|f(z)|\leq M$ for $z$ interior to $\Gamma_\varrho$. Then, there exists a constant $C>0$, independent of $f$, such that for all $\ell\in\mathbb{N}$,
\begin{equation*}
\left|\langle f,p_\ell\rangle_{L^2_w(\Sigma)}\right|\leq CM\varrho^{-\ell}.
\end{equation*}
\end{lemma}
\begin{proof}
Define $\Omega_\eta=\{z\in\compl~:~\varrho^{\eta}\leq\ex^{\re \mathfrak g(z)}\leq \varrho\}$ for some $0<\eta<1$. Choose $\epsilon$ in Lemma \ref{poly_asymp} to be sufficiently small such that $V\in\Omega_\eta^c$, and denote the uniform bound on $\left|\delta_\ell(z)\right|$ by $D$. Let $\varrho'$ satisfy $1<\varrho^{\eta}<\varrho'<\varrho$, and consider the level curve $\Gamma_{\varrho'}$. Then, by Cauchy's integral formula and Lemma~\ref{poly_asymp},
\begin{equation*}
	\left|\langle f,p_\ell\rangle_{L^2_w(\Sigma)}\right|=\left|\int_\Sigma\frac{1}{2\pi\im}\int_{\Gamma_{\varrho'}}\frac{f(z)}{z-x}\df zp_\ell(x)w(x)\df x\right|=\left|-\int_{\Gamma_{\varrho'}}f(z)\mathcal{C}_\Sigma\left[p_\ell w\right](z)\df z\right|\leq\left|\Gamma_{\varrho}\right|M\frac{D}{2\pi}\varrho'^{-\ell},
\end{equation*}
using $C=\frac{1}{2\pi}\left|\Gamma_{\varrho}\right|D$. Since this inequality holds for all $\varrho' < \varrho$, it must hold for $\varrho$.
\end{proof}

\begin{theorem}\label{t:nearsing}
Under the assumptions of Lemma \ref{bernstein}, suppose that
\begin{align*}
    \sum_{\ell=0}^{\infty}\varrho^{-\ell}\|p_\ell(\bA)\bb\|_2 < \infty.
\end{align*}
Then, there exist $(\varepsilon_{m,\ell})$ satisfying $\left|\varepsilon_{m,\ell}\right|\leq\|\rho_m\|_\infty := \max_{x \in \Sigma} |\rho_m(x)|$ and a constant $C>0$ such that in any norm
\begin{equation*}
\left\|f(\bA)\bb-f_{k,m}(\bA)\bb+\sum_{\ell=0}^{k-1}\varepsilon_{m,\ell}p_\ell(\bA)\bb\right\|\leq CM\sum_{\ell=k}^{\infty}\varrho^{-\ell}\|p_\ell(\bA)\bb\|.
\end{equation*}
\end{theorem}

\begin{proof}
For a given $\ell$,
\begin{equation*}
	-\sum_{j=1}^mf(z_j)w_j\mathcal{C}_\Sigma\left[p_\ell w\right](z_j)=\int_\Sigma p_\ell(x)\frac{1}{2\pi\im}\sum_{j=1}^m\frac{f(z_j)w_j}{z_j-x}w(x)\df x=\langle f_m,p_\ell\rangle_{L^2_w(\Sigma)},
\end{equation*}
where 
\begin{equation*}
	f_m(x) = \frac{1}{2\pi\im}\sum_{j=1}^m\frac{f(z_j)w_j}{z_j-x}\approx f(x).
\end{equation*}
Thus,
\begin{equation*} f_{k,m}=\sum_{\ell=0}^{k-1}\langle f_m,p_\ell\rangle_{L^2_w(\Sigma)}p_\ell.
\end{equation*}
Let $\varepsilon_{m,\ell}=\langle f_m-f,p_\ell\rangle_{L^2_w(\Sigma)}$. By the Cauchy--Schwarz inequality,
\begin{equation*}
\left|\varepsilon_{m,\ell}\right|\leq \|f_m-f\|_{L^2_w(\Sigma)}\leq\|\rho_m\|_\infty.
\end{equation*}
Then, 
\begin{equation*}
f(\bA)\bb-f_{k,m}(\bA)\bb+\sum_{\ell=0}^{k-1}\varepsilon_{m,\ell}p_\ell(\bA)\bb=\sum_{\ell=k}^{\infty}\langle f,p_\ell\rangle_{L^2_w(\Sigma)}p_\ell(\bA)\bb.
\end{equation*}
By the triangle inequality and Lemma \ref{bernstein}, there exists some $C>0$ such that
\begin{equation*}
\left\|f(\bA)\bb-f_{k,m}(\bA)\bb+\sum_{\ell=0}^{k-1}\varepsilon_{m,\ell}p_\ell(\bA)\bb\right\|\leq CM\sum_{\ell=k}^{\infty}\varrho^{-\ell}\|p_\ell(\bA)\bb\|.
\end{equation*}
\end{proof}

The bound from Theorem~\ref{t:nearsing} implies that when $\bA$ is diagonalized as $\bA = \bV \bLambda \bV^{-1}$,
\begin{align*}
 \|f(\bA)\bb - f_{k,m}(\bA)\bb\|_2 \leq \|\bV\|_2 \|\bV^{-1}\|_2 \| \bb \|_2 \left[  \|\rho_m\|_\infty \sum_{\ell = 0}^{k-1} \|p_\ell (\bLambda)\|_2 + C M \sum_{\ell = k}^\infty \varrho^{-\ell} \|p_\ell (\bLambda)\|_2 \right].
\end{align*}
This demonstrates that if the quadrature rule is taken such that $\|\rho_m\|_\infty$ is sufficiently small, then the convergence rate will be largely determined by the distance from $\Sigma$ to the nearest singularity of $f$.  More precisely, $\|f(\bA)\bb-f_{k,m}(\bA)\bb\| = \OO(\varrho^{-k} + \|\rho_m\|_\infty)$ when the spectrum of $\bA$ is contained in $\Sigma$, and the rate of convergence will be governed by the furthest level curve $\Gamma_\varrho$ to which $f$ can be analytically extended. If $f$ is entire, the initial convergence is superexponential. The level curves of Lemma \ref{bernstein} can be thought of as higher genus analogs to Bernstein ellipses as they dictate the convergence rate in the same manner. A contour plot of $\ex^{\re \mathfrak g(z)}$ is included as part of Figure \ref{error_and_contour}.

Because Theorem~\ref{thm:func_error} gives a different bound, we have, supposing that $\Gamma$ lies inside $\Gamma_\varrho$,
\begin{align*}
    \frac{\|f(\bA)\bb - f_{k,m}(\bA)\bb\|_2}{\|\bV\|_2 \|\bV^{-1}\|_2 \| \bb \|_2} &\leq \min\left\{  \|\rho_m\|_\infty \sum_{\ell = 0}^{k-1} \|p_\ell (\bLambda)\|_2 + C M \sum_{\ell = k}^\infty \varrho^{-\ell} \|p_\ell (\bLambda)\|_2,\right. \\
    &\left. \qquad \frac{1}{2\pi}\|\rho_m(\bLambda)\|_2 +cLM\max_{j\in\{1,\ldots,m\}}\frac{\ex^{k\nu(z_j;\bA)}}{1-\ex^{\nu(z_j;\bA)}},\right\}\\
    &\leq \min\left\{  c' \|\rho_m\|_\infty \sum_{\ell=0}^{k-1} \ex^{\ell\nu(\bA)} + c' C M \frac{\ex^{k \nu(z_*;\bA)}}{1 - \ex^{\nu(z_*;\bA)}},\right. \\
    &\left. \qquad \frac{1}{2\pi}\|\rho_m(\bLambda)\|_2 +cLM\max_{j\in\{1,\ldots,m\}}\frac{\ex^{k\nu(z_j;\bA)}}{1-\ex^{\nu(z_j;\bA)}},\right\},
\end{align*}
for another constant $c'> 0$ by Lemma~\ref{poly_asymp}.  Here $z_*$ is any point on $\Gamma_\varrho$.

If $\bA$ is generic and the spectrum of $\bA$ is contained in $\Sigma$, then $\nu(\bA) = 0$ ($\|p_\ell(\bA)\|_2 = O(1)$) and the first term inside the minimum is likely to be smaller than the second since the nodes $z_j$ are inside $\Gamma_\varrho$, but, eventually, the first term will dominate the second when $\sum_{\ell = 0}^{k-1} \|p_\ell(\bLambda)\|_2$ becomes large. The second term then implies that the errors remain small, even if this sum is large.  Also, if $\bA$ is generic but the eigenvalues of $\bA$ are not contained in $\Sigma$, the second term here will still be small, giving useful bounds.  The first term will also likely be small, but as $\nu(\bA) > 0$, the sum grows exponentially with respect to $k$. In this case, the balance, i.e., which term gives a better bound in a given situation, will change. 

\subsection{Complexity}

To run Algorithm~\ref{alg:akh_func} for $k$ steps, using that computing $a_j, b_j, S_j$ requires $\OO(1)$ arithmetic operations and supposing that $T_\bA$ is the complexity of matrix multiplication by $\bA$, we need less than 
\begin{align*}
k T_\bA + 7 k n + k m + \OO((1 + m)k)
\end{align*}
arithmetic operations if $\bA \in \mathbb C^{n \times n}$.  Here the $\OO((1+m)k)$ term represents the computation of $a_j, b_j, S_j$ for $j = 0,1,\ldots k$. The complexity of the first Akhiezer iteration, Algorithm~\ref{alg:akh_iter}, is found by setting $m = 0$.  Furthermore, if a tolerance of $\epsilon$ is desired for function approximation, we expect to set $m = \OO(\log \epsilon^{-1}).$  This is in comparison with the method in \cite{SaadMatrixFunction} where the complexity is 
\begin{align*}
    k T_\bA + \OO(k^2 m  + n),
\end{align*}
and $m = \OO(\epsilon^{2/7})$ is required, in general, to achieve a tolerance of $\epsilon$.


\section{Adaptivity}\label{sect:adapt}
In the case where the matrix $\bA$ is symmetric, i.e., it has all real eigenvalues, the asymptotic behavior of the polynomials yields an adaptive algorithm that estimates a good choice for the bands $\Sigma=\bigcup_{j=1}^{g+1}[a_j,b_j]$ then applies the iteration \eqref{eq:iteration} to solve the system $\bA\bx=\bb$. For simplicity, we will consider only the $g=1$ case $\Sigma=[a_1,b_1]\cup[a_2,b_2]$, but such an approach can be applied to the $g>1$ cases as well. Furthermore, such an approach could be used as a first step in computing a matrix function with the first step being producing an estimate of the spectrum of $\bA$.

By Lemma \ref{poly_asymp}, if not all eigenvalues of $\bA$ are contained in $\Sigma$, the matrix polynomials $p_j(\bA)$, and therefore $\|p_j(\bA)\bb\|$ as well, grow at an asymptotic rate given by $\ex^{\max_k\re\mathfrak g(\lambda_k)}$, where $\lambda_1,\ldots,\lambda_n$ denote the eigenvalues of $\bA$. Thus, if the approximate growth rate is given by $r$, then performing root finding on $\ex^{\re g(z)}-r$ around each endpoint of $\Sigma$ will yield a new set of bands that contain the eigenvalues of $\bA$. A naive implementation of this is given as Algorithm \ref{alg:adapt}.
\begin{algorithm}
	\caption{Adaptive band selection}\label{alg:adapt}
	\textbf{Input: }{Initial guess for bands $\Sigma=[a_1,b_1]\cup[a_2,b_2]$ where $b_1<0<a_2$, $\bA\in\compl^{n\times n}$, $\bA = \bA^*$, $\bb\in\compl^n$, parameters $\gamma_o>1$, $0<\gamma_i<1$.}\\
	\While{\emph{Norms of matrix polynomials $\|p_j(\bA)\bb\|$ grow as $j\to\infty$}}{
		Compute the growth rate $r$ of $\|p_j(\bA)\bb\|$ as $j\to\infty$.\\
		Set $f(x)=\ex^{\re \mathfrak g(x)}-r$ for $\mathfrak g$ corresponding to $[a_1,b_1]\cup[a_2,b_2]$.\\
		Perform bisection on $f$ on $[\gamma_oa_1,a_1]$. If bisection fails, set $a_1\gets\gamma_oa_1$; otherwise set $a_1\gets c$ where $c$ is the computed root.\\
		Perform bisection on $f$ on $[b_1,\gamma_ib_1]$. If bisection fails, set $b_1\gets\gamma_ib_1$; otherwise set $b_1\gets c$ where $c$ is the computed root.\\
		Perform bisection on $f$ on $[\gamma_ia_2,a_2]$. If bisection fails, set $a_2\gets\gamma_ia_2$; otherwise set $a_2\gets c$ where $c$ is the computed root.\\
		Perform bisection on $f$ on $[b_2,\gamma_ob_2]$. If bisection fails, set $b_2\gets\gamma_ob_2$; otherwise set $b_2\gets c$ where $c$ is the computed root.\\
		Compute norms of matrix polynomials $\|p_j(\bA)\bb\|$ for $p_0,p_1,\ldots$ corresponding to the updated bands $\Sigma=[a_1,b_1]\cup[a_2,b_2]$.
	}
	\Return{$\Sigma=[a_1,b_1]\cup[a_2,b_2]$}
\end{algorithm}
This algorithm ensures that the Akhiezer iteration on the outputted bands $\Sigma=[a_1,b_1]\cup[a_2,b_2]$ applied to $\bA\bx=\bb$ will converge at the rate $\ex^{-\re\mathfrak g(0)}$. In practice, computing the growth rate of the matrix polynomials can be challenging due to theta function oscillations implicit in the orthogonal polynomials. One way to do this is to consider the ratio
\begin{equation*}
	\frac{|p_{n+k}(z)|}{|p_n(z)|}=\frac{|\hat\delta_{n+k}(z)|}{|\hat\delta_{n}(z)|}\ex^{k\re\mathfrak g(z)},
\end{equation*}
which follows from Lemma \ref{poly_asymp}. Since $\hat\delta_n(z)=\OO(1)$, for sufficiently large $n$ and $k$,
\begin{equation*}
	r\approx\ex^{\re\mathfrak g(z)}\approx\sqrt[k]{\frac{|p_{n+k}(z)|}{|p_n(z)|}}.
\end{equation*}
Of course, what ``sufficiently large'' means in practice is not always clear, and one may need to experiment with different values of $n$ and $k$ for different matrices.

The factor $\ex^{-\re\mathfrak g(0)}$ typically increases as the sizes of the intervals increase, so the convergence rate will be suboptimal if the eigenvalues of $\bA$ are contained in some $\Sigma'\subsetneq\Sigma$. The convergence rate can be improved by modifying Algorithm \ref{alg:adapt} so that only one endpoint is moved at a time and the new growth rate is computed after each movement. If the growth rate decreases, the algorithm proceeds to the next endpoint. Otherwise, the endpoint is reverted to its previous value and the algorithm proceeds.

An improved band estimation technique can be employed if one is willing up to compute 4 (or more generally, $2g+2$) inner products. This is done by approximating the eigenvector corresponding to the eigenvalue with the most growth. In particular, one computes $\Tilde\by=p_n(\bA)\bb$ for a sufficiently large $n$ and normalizes $\by=\frac{\Tilde\by}{\|\Tilde{\by}\|}$ to avoid round-off errors. This is iterated as $\Tilde\by=p_n(\bA)\by$, $\by=\frac{\Tilde\by}{\|\Tilde{\by}\|}$ until convergence. Then, the eigenvalue corresponding to the largest growth rate is approximated by the Rayleigh quotient
\begin{equation*}
	\frac{\by^T\bA\by}{\by^T\by}.
\end{equation*}
The relevant endpoint is then moved to strictly contain this eigenvalue, and the algorithm iterates until all approximated eigenvalues fall within the bands. We find this approach to be the most accurate, and the eigenvalues are computed to machine precision, so the algorithm needs only 4 Rayleigh quotient computations to accurately approximate the true bands.

On the other hand, if one only cares that the iteration converges and is not concerned with the rate, a simpler approach can be applied. If one takes the initial guess for the bands to be the symmetric intervals $[-b,-a]\cup[a,b]$, convergence can be obtained without moving the interior endpoints. This follows from the fact that the iteration converges if $\re\mathfrak g(\lambda_k)<\re\mathfrak g(0)$ for all eigenvalues $\lambda_k$ of $\bA$ and that $\re\mathfrak g(z)$ is maximized on $[-a,a]$ at $z=0$. Thus, one only needs to adjust the parameter $b$ so that $\re\mathfrak g(\lambda_k)<\re\mathfrak g(0)$ for all eigenvalues such that $|\lambda_k|>b$.

\section{Examples and Applications}\label{sect:examples}
Since this work is inspired by that of Saad, we first consider an example presented in \cite[Section 8.1]{Saad}. Let $\bA\in\real^{200\times200}$ be a diagonal matrix with eigenvalues spaced uniformly in $[-2,-0.5]\cup[0.5,6]$, and let $\bb=\bA\be$ where $\be\in\real^{200}$ is the vector of all ones. We run the Akhiezer iteration using $g=1$ Akhiezer polynomials and plot the relative residuals in Figure \ref{saad_example}. We observe that this looks similar to the result in \cite{Saad}.
\begin{figure}
	\includegraphics[scale=0.5]{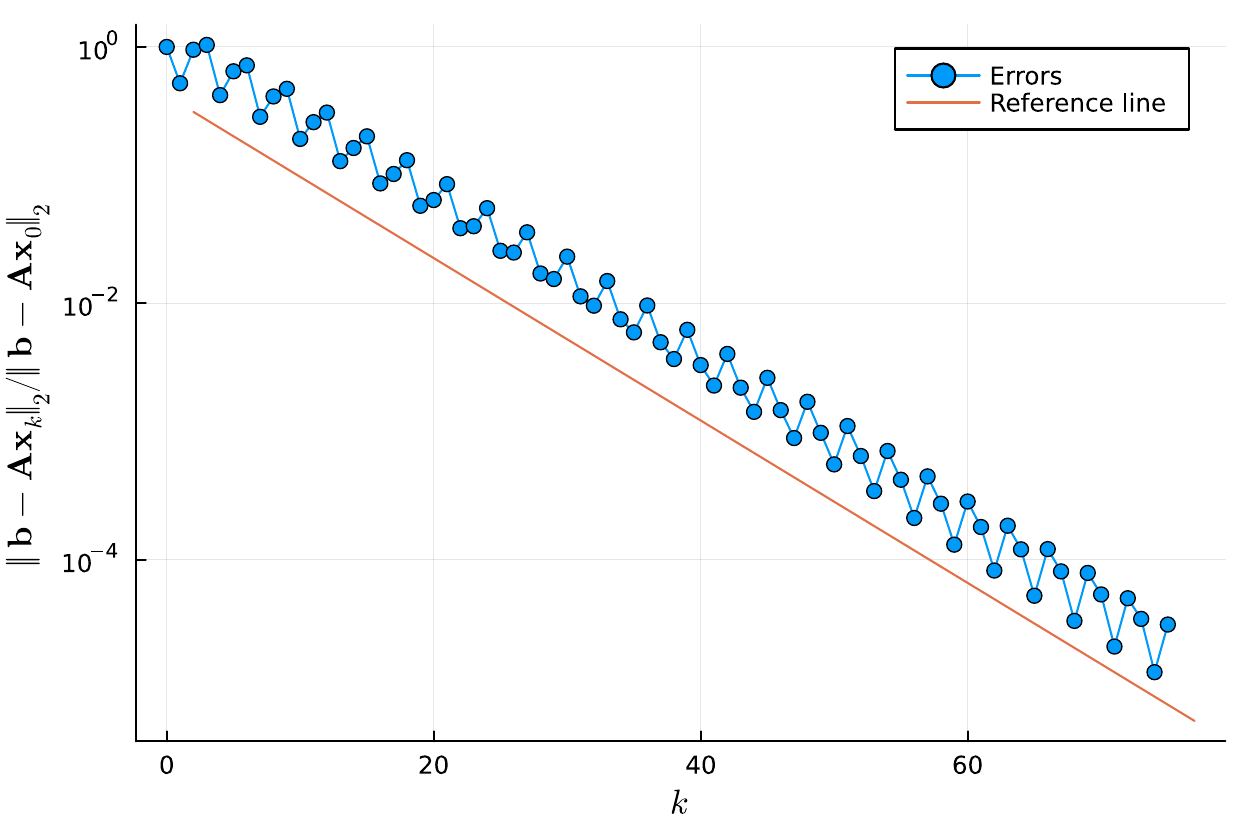}
	\caption{Relative residuals at each iteration of the Akhiezer iteration applied to a diagonal matrix $\bA\in\real^{200\times200}$ with eigenvalues spaced uniformly in $[-2,-0.5]\cup[0.5,6]$ and $\bb=\bA\be$. The rate of convergence is well-approximated by the reference line $C\ex^{-k\re \mathfrak{g}(0)}.$}
	\label{saad_example}
\end{figure}

To construct a more complicated example, we consider a $200\times200$ real matrix with randomly generated eigenvalues in $[-2,-0.5]\cup[0.5,6]$. We construct $\bA$ by moving the eigenvalues off the intervals by adding a Gaussian perturbation and removing the near-diagonal structure by conjugating by an orthogonal matrix. Again, we take $\bb=\bA\be$. In Figure \ref{error_and_contour}, we plot the eigenvalues of $\bA$ on a contour map of $\re\mathfrak g$ and the relative residuals for the first 370 iterations of our the Akhiezer iteration using $g=1$ Akhiezer polynomials.
\begin{figure}
	\centering
	\begin{subfigure}{0.495\linewidth}
		\centering
		\includegraphics[width=\linewidth]{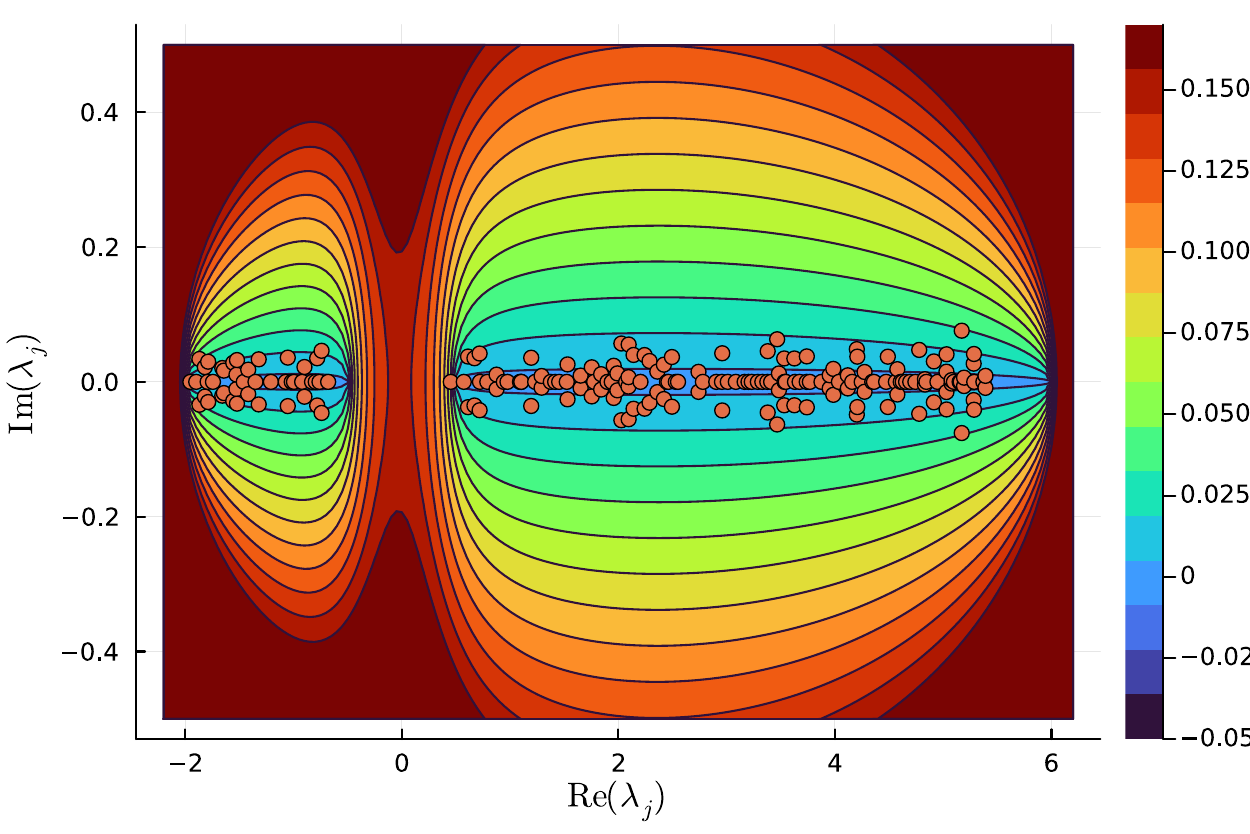}
	\end{subfigure}
	\begin{subfigure}{0.495\linewidth}
		\centering
		\includegraphics[width=\linewidth]{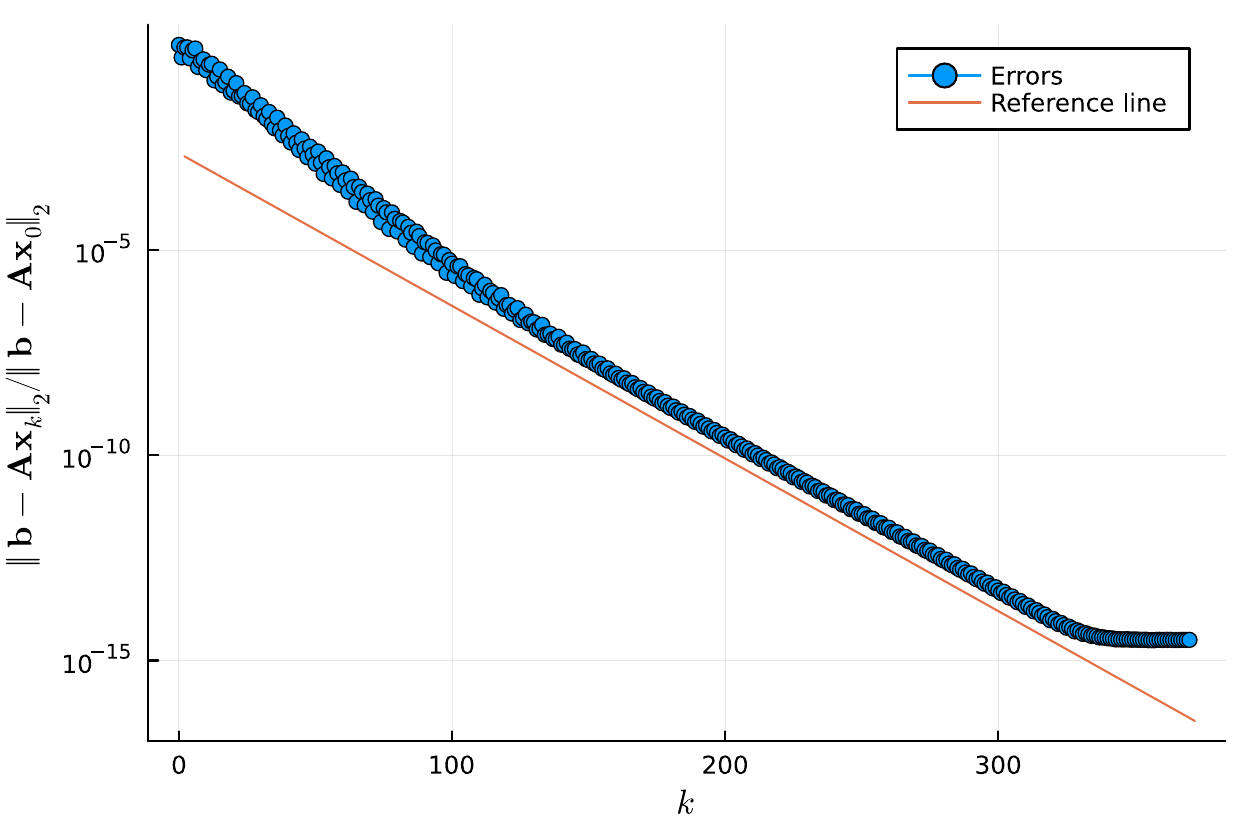}
	\end{subfigure}
	\caption{Eigenvalues of a matrix $\bA$ superimposed on a contour plot of $\re\mathfrak g(z)$ (left) and relative residuals at each iteration for the Akhiezer iteration applied to $\bA\bx=\bb:=\bA\be$. The asymptotic rate of convergence is well-approximated by the reference line $C\ex^{k\nu(0;\bA)}.$}
	\label{error_and_contour}
\end{figure}

To demonstrate the convergence of the Akhiezer iteration for matrix function approximation, we consider the same matrix $\bA$ as the previous example but let $\bb\in\real^{200}$ be a random Gaussian vector. We let the contour be two circles: one centered around each interval with diameter 1.15 times the length of the interval. We take 200 quadrature nodes divided between the two circles proportionally to their width. First considering the entire function $f(x)=\exp(x)$, we use \eqref{mat_func_iter} with $g=1$ Akhiezer polynomials to approximate $f(\bA)\bb$ and plot the 2-norm of the difference between this approximation and $f(\bA)\bb$ in Figure \ref{mat_func_conv_exp}. We observe superexponential convergence until the iteration saturates due to rounding errors. Instead, letting $f(x)=\tanh(x)$ and then $f(x)=\exp(x)/x$, we plot the same errors in at each iteration in Figures~\ref{mat_func_conv_sym} and \ref{mat_func_conv} for two different matrices $\bA$. We observe that this approach allows one to compute $f(\bA)\bb$ for functions $f$ that have singularities at zero, albeit at a slower rate than for functions with singularities farther away from $\Sigma$. Furthermore, the bound of Theorem \ref{t:nearsing} estimates the convergence rate accurately when the spectrum of $\bA$ is contained in $\Sigma$ and is pessimistic but close when this is not true.
\begin{figure}
	\centering
		\includegraphics[scale=0.5]{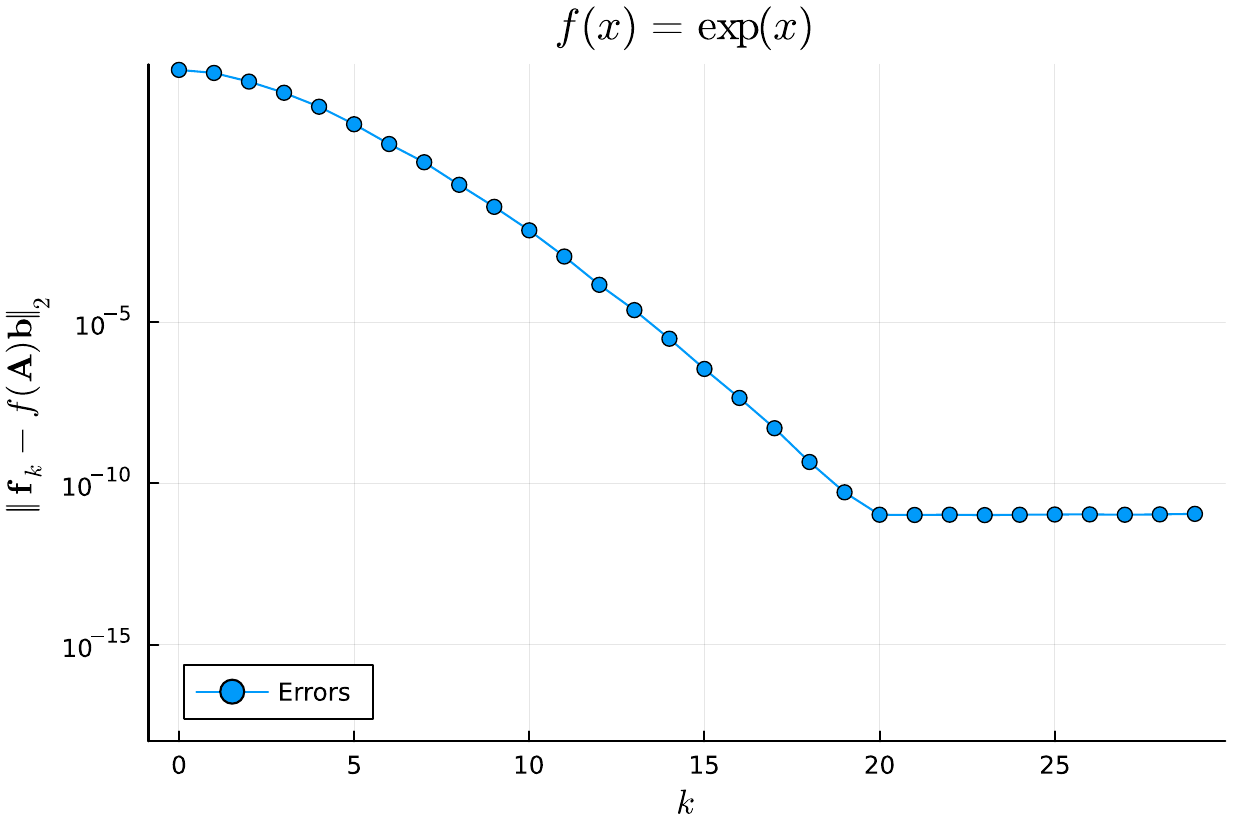}
	\caption{2-norm error at each iteration for the Akhiezer iteration approximation to $f(\bA)\bb$ where $f(x)=\exp(x)$, $\bA\in\real^{200\times200}$ is as in Figure \ref{error_and_contour}, and $\bb\in\real^{200}$ is a random Gaussian vector. The iteration converges superexponentially, saturating within 20 iterations.}
	\label{mat_func_conv_exp}
\end{figure}
\begin{figure}
	\centering
	\begin{subfigure}{0.495\linewidth}
		\centering
		\includegraphics[width=\linewidth]{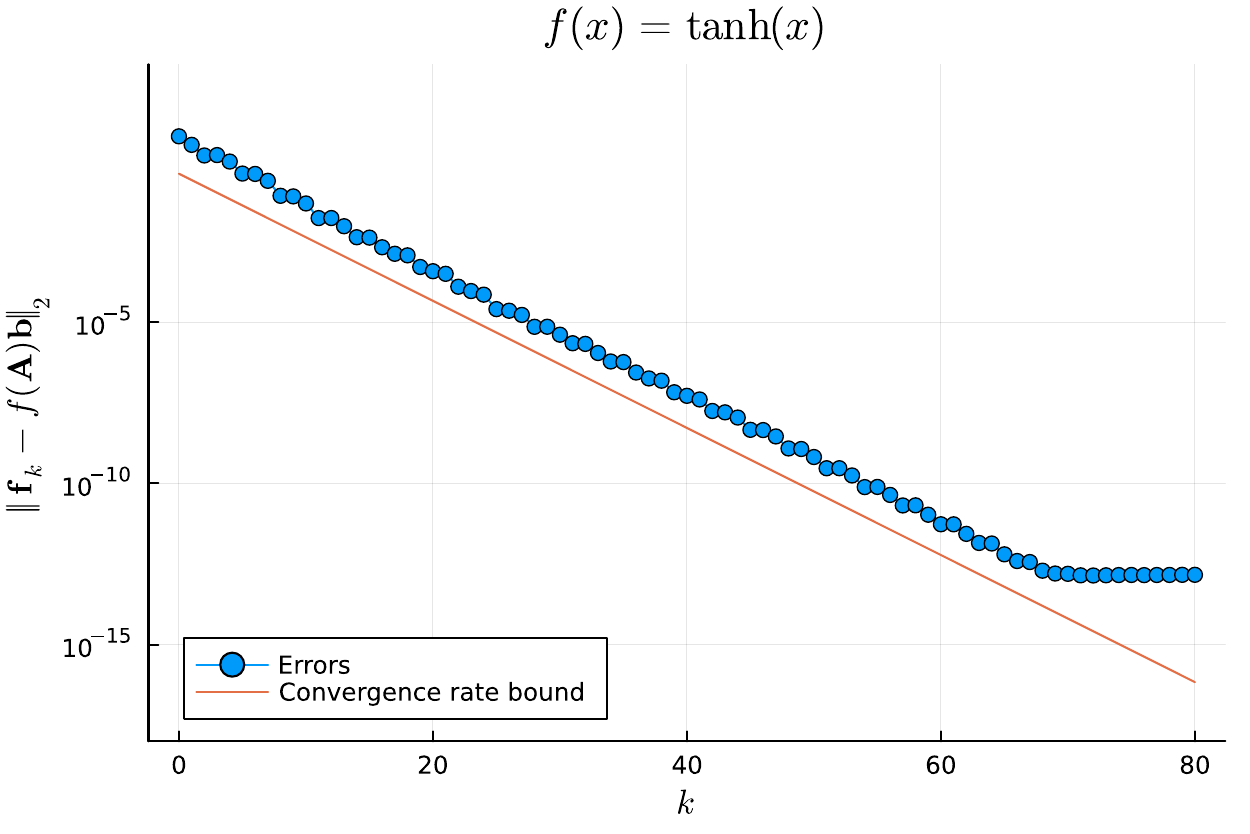}
	\end{subfigure}
	\begin{subfigure}{0.495\linewidth}
		\centering
		\includegraphics[width=\linewidth]{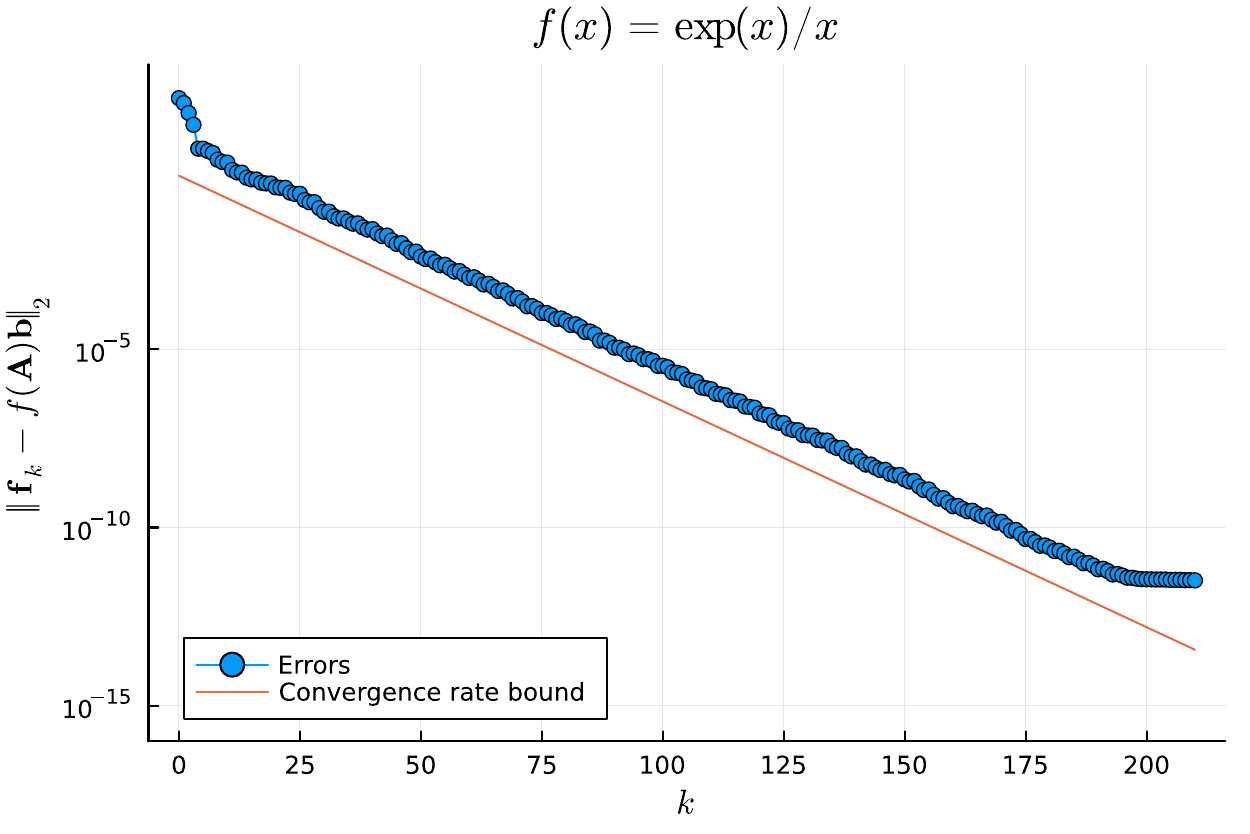}
	\end{subfigure}
	\caption{2-norm error at each iteration for the Akhiezer iteration approximation to $f(\bA)\bb$ where $f(x)=\tanh(x)$ (left) and $f(x)=\exp(x)/x$ (right), $\bA\in\real^{200\times200}$ is has randomly generated eigenvalues in $[-2,-0.5]\cup[0.5,6]$, and $\bb\in\real^{200}$ is a random Gaussian vector. The convergence rate bounds $C_1\ex^{k\nu(\im\pi/2;\bA)}$ (left) and $C_2\ex^{k\nu(0;\bA)}$ (right) approximate the true convergence rates.}
	\label{mat_func_conv_sym}
\end{figure}
\begin{figure}
	\centering
	\begin{subfigure}{0.495\linewidth}
		\centering
		\includegraphics[width=\linewidth]{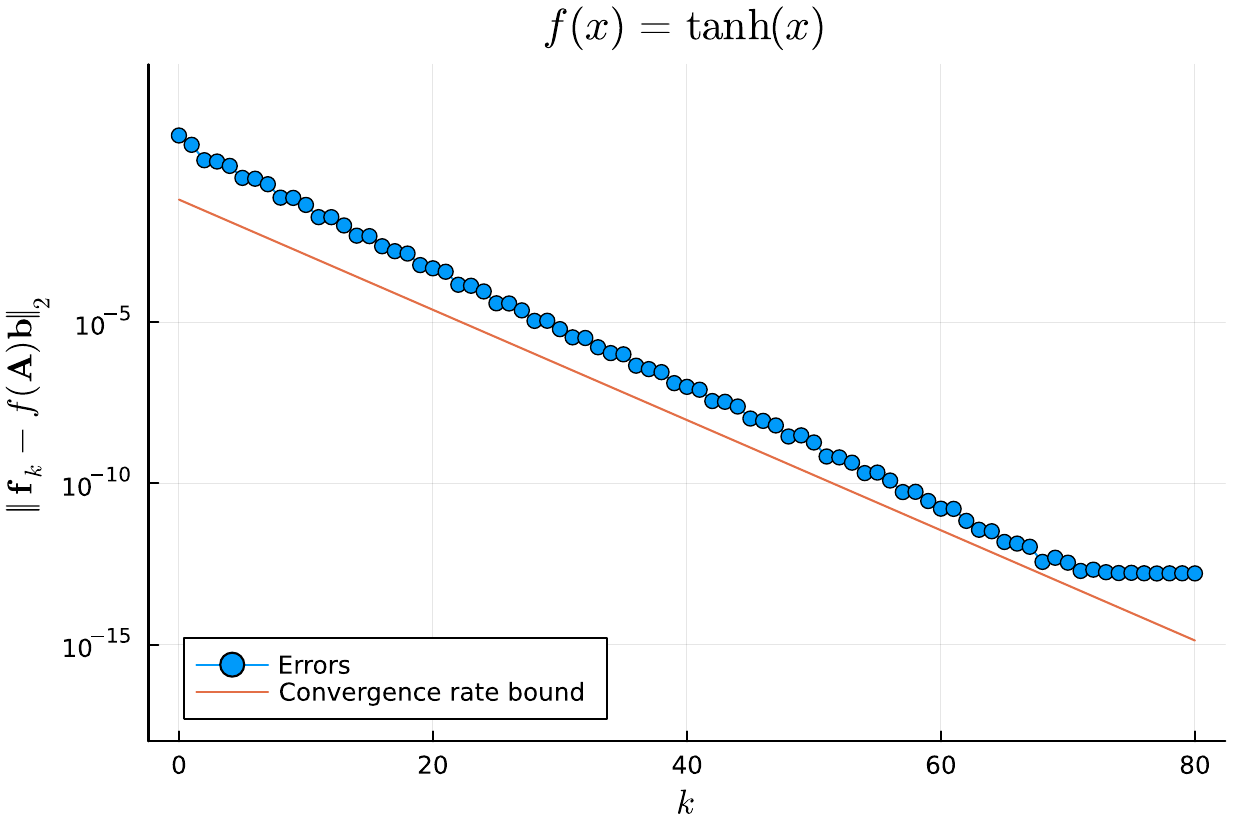}
	\end{subfigure}
	\begin{subfigure}{0.495\linewidth}
		\centering
		\includegraphics[width=\linewidth]{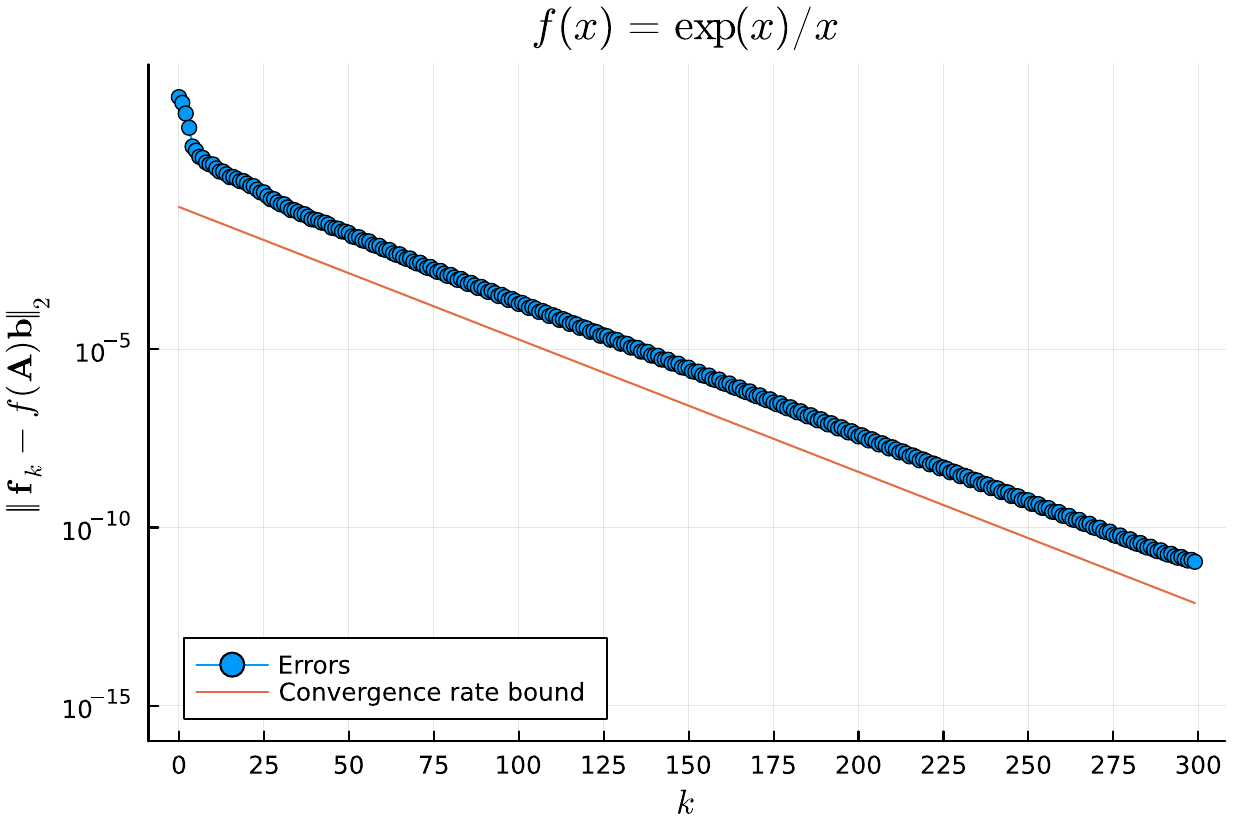}
	\end{subfigure}
	\caption{2-norm error at each iteration for the Akhiezer iteration approximation to $f(\bA)\bb$ where $f(x)=\tanh(x)$ (left) and $f(x)=\exp(x)/x$ (right), $\bA\in\real^{200\times200}$ is as in Figure \ref{error_and_contour}, and $\bb\in\real^{200}$ is a random Gaussian vector. The convergence rate bounds $C_1\ex^{k\nu(\im\pi/2;\bA)}$ (left) and $C_2\ex^{k\nu(0;\bA)}$ (right) approximate the true convergence rates.}
	\label{mat_func_conv}
\end{figure}

To illustrate the advantage of having more knowledge of the spectrum, we consider $\bA$ and $\bb$ constructed in the same manner as the previous example but with the original eigenvalues taken to be in $\Sigma=[-2,-0.5]\cup[0.5,0.7]\cup[5.8,6]$. In Figure \ref{2vs3}, we plot the convergence of the Akhiezer iteration using Akhiezer polynomials on $[-2,-0.5]\cup[0.5,6]$ and orthogonal polynomials with the weight \eqref{special_weight} on $\Sigma$. We find that the iteration converges to a relative residual of $10^{-10}$ in 177 iterations in the former case and 102 in the latter in spite of the unbounded behavior at the endpoints of orthogonal polynomials with the weight \eqref{special_weight}.
\begin{figure}
	\centering
	\begin{subfigure}{0.495\linewidth}
		\centering
		\includegraphics[width=\linewidth]{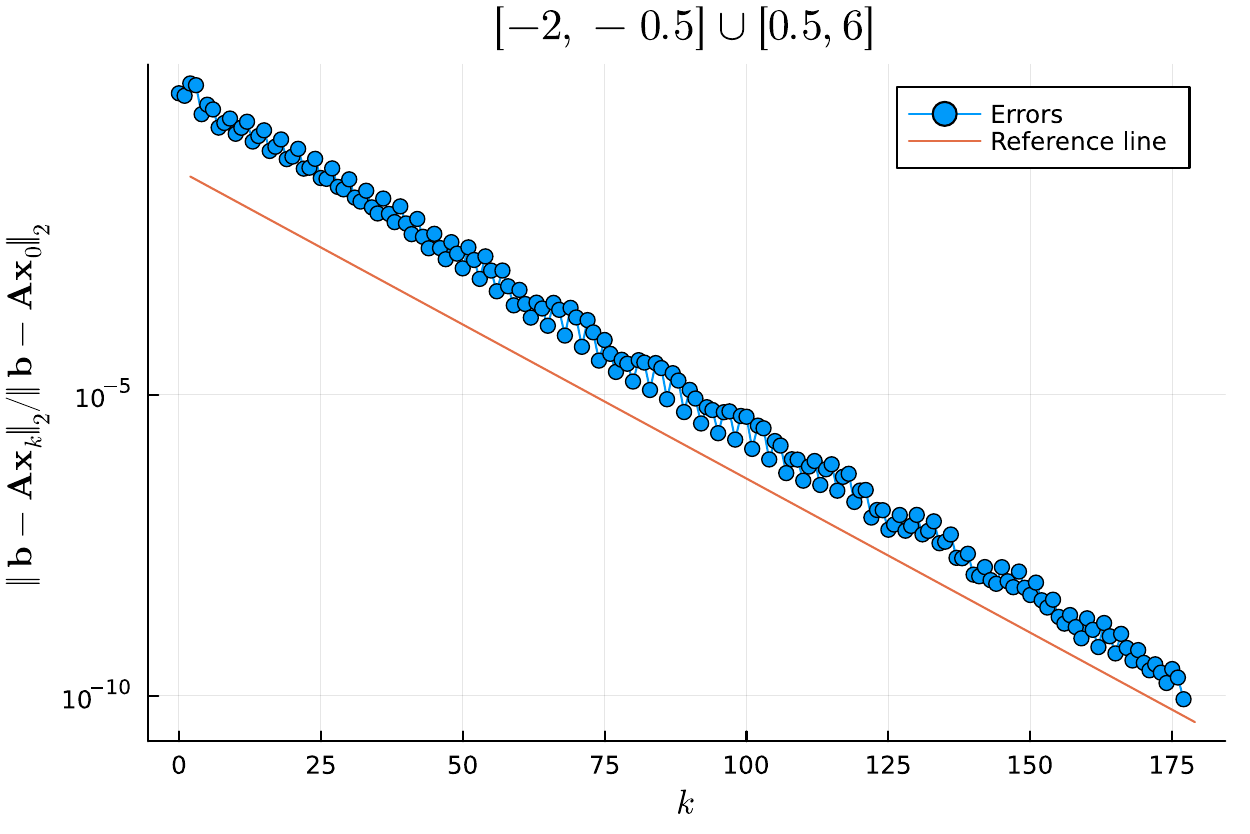}
	\end{subfigure}
	\begin{subfigure}{0.495\linewidth}
		\centering
		\includegraphics[width=\linewidth]{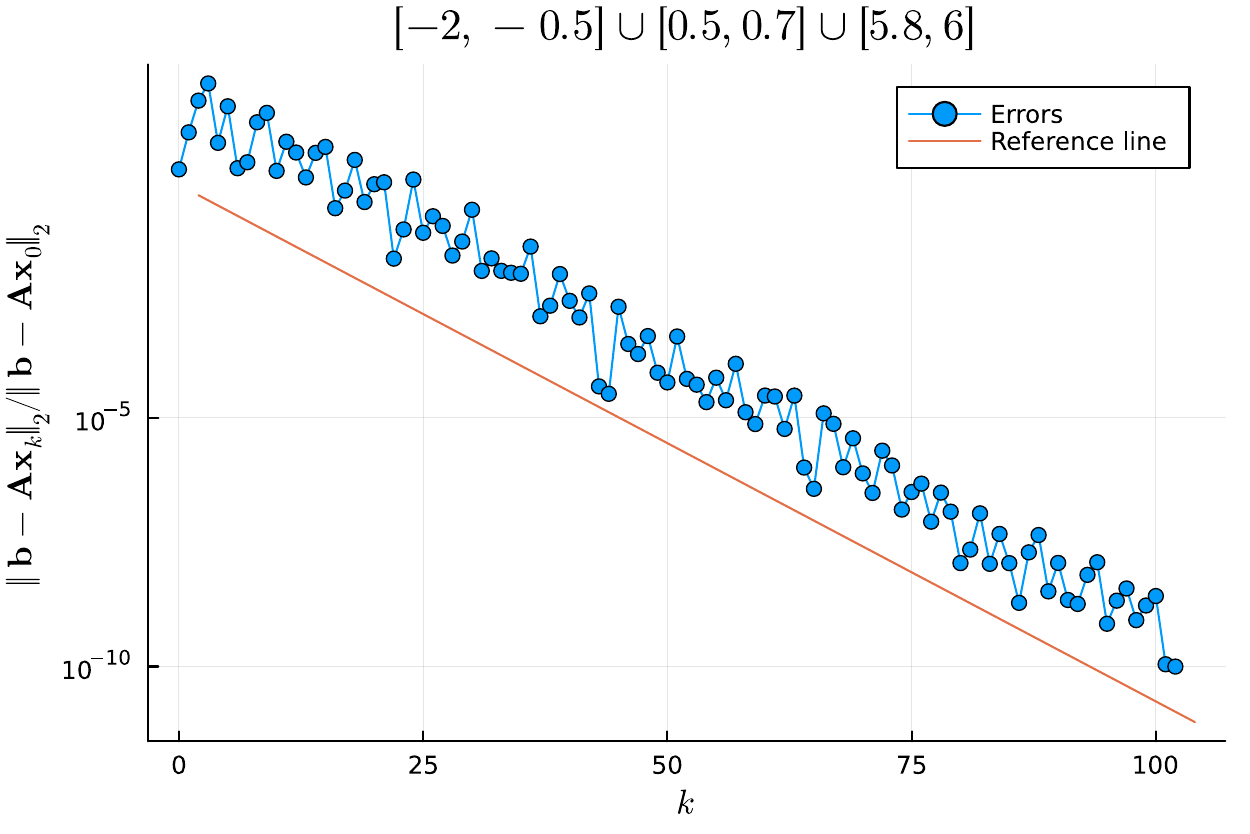}
	\end{subfigure}
	\caption{Relative residuals at each iteration for the Akhiezer iteration applied to $\bA\bx=\bb$ where $\bA\in\real^{200\times200}$ has eigenvalues near $[-2,-0.5]\cup[0.5,0.7]\cup[5.8,6]$ and $\bb\in\real^{200}$ is a random Gaussian vector with Akhiezer polynomials on $[-2,-0.5]\cup[0.5,6]$ (left) and and orthogonal polynomials with the weight \eqref{special_weight} on $[-2,-0.5]\cup[0.5,0.7]\cup[5.8,6]$ (right). The convergence rate $\ex^{\nu(0;\bA)}$ plotted in the reference line is approximately 0.888 (left) and 0.787 (right). The iteration converges to a relative residual of $10^{-10}$ in 177 iterations (left) and 102 iterations (right).}
	\label{2vs3}
\end{figure}

As a final example, we consider the boundary-value problem (BVP) on $[0,1]$:
\begin{equation*}
\begin{aligned}
-\frac{\df^2}{\df x^2}u(x)-30\ex^xu(x)=x,\\
u(0)=u(1)=0.
\end{aligned}
\end{equation*}
Applying a centered finite difference approximation on an equally spaced grid of 100 points yields a symmetric indefinite linear system $\bA\bx=\bb$ with eigenvalues on the positive real axis that are unbounded as the grid spacing is refined. To control the eigenvalues, we apply a preconditioner $\bL^{-1}$ where $\bL$ corresponds to a centered finite difference approximation on the same grid applied to the BVP
\begin{equation*}
	\begin{aligned}
		-\frac{\df^2}{\df x^2}u(x)=0,\\
		u(0)=u(1)=0.
	\end{aligned}
\end{equation*}
We then expect $\bL^{-1}\bA$ to have the majority of its eigenvalues bounded by 1 on the positive real axis with the remaining eigenvalues on the negative real axis. Thus, we apply Algorithm \ref{alg:adapt} to the system $\bL^{-1}\bA\bx=\bL^{-1}\bb$ with an initial guess of $\Sigma=[-2,-0.5]\cup[0.5, 1]$ and parameters $\gamma_o=5,\gamma_i=0.7$ and obtain the bands $\Sigma=[-4.16236,-0.24854]\cup[0.25104, 3.10107]$. The true eigenvalues of $\bL^{-1}\bA$ fall in $[-4.14928,-0.28169]\cup[0.43062,0.99921]$ with eigenvalues at each endpoint, so this algorithm does a good job of capturing the ``worst offender'' but far overshoots the right endpoint. However, applying the modification of Algorithm \ref{alg:adapt} to optimize the convergence rate described in Section \ref{sect:adapt} yields the bands $\Sigma=[-4.15388,-0.28391]\cup[0.44168, 1.01575]$; these are very close to the true bands, but note that the interval endpoint $-0.28391 < -0.28169$, and an eigenvalue lies outside the computed bands. This occurs due to difficulties in approximating the exact matrix polynomial growth rate ($r$ in Algorithm~\ref{alg:adapt}) but gives just a slight perturbation to the geometric rate of convergence. Furthermore, applying the variation with inner products finds the true bands to full precision.

In Figure \ref{schro}, we plot the relative residuals when applying the Akhiezer iteration to $\bL^{-1}\bA\bx=\bL^{-1}\bb$ with the Akhiezer polynomials on $\Sigma=[-4.14823,-0.245]\cup[0.245, 3.09077]$ and the Akhiezer polynomials on $\Sigma=[-4.15746,-0.30738]\cup[0.42308, 1.01751]$ for 250 iterations.
\begin{figure}
	\centering
	\begin{subfigure}{0.495\linewidth}
		\centering
		\includegraphics[width=\linewidth]{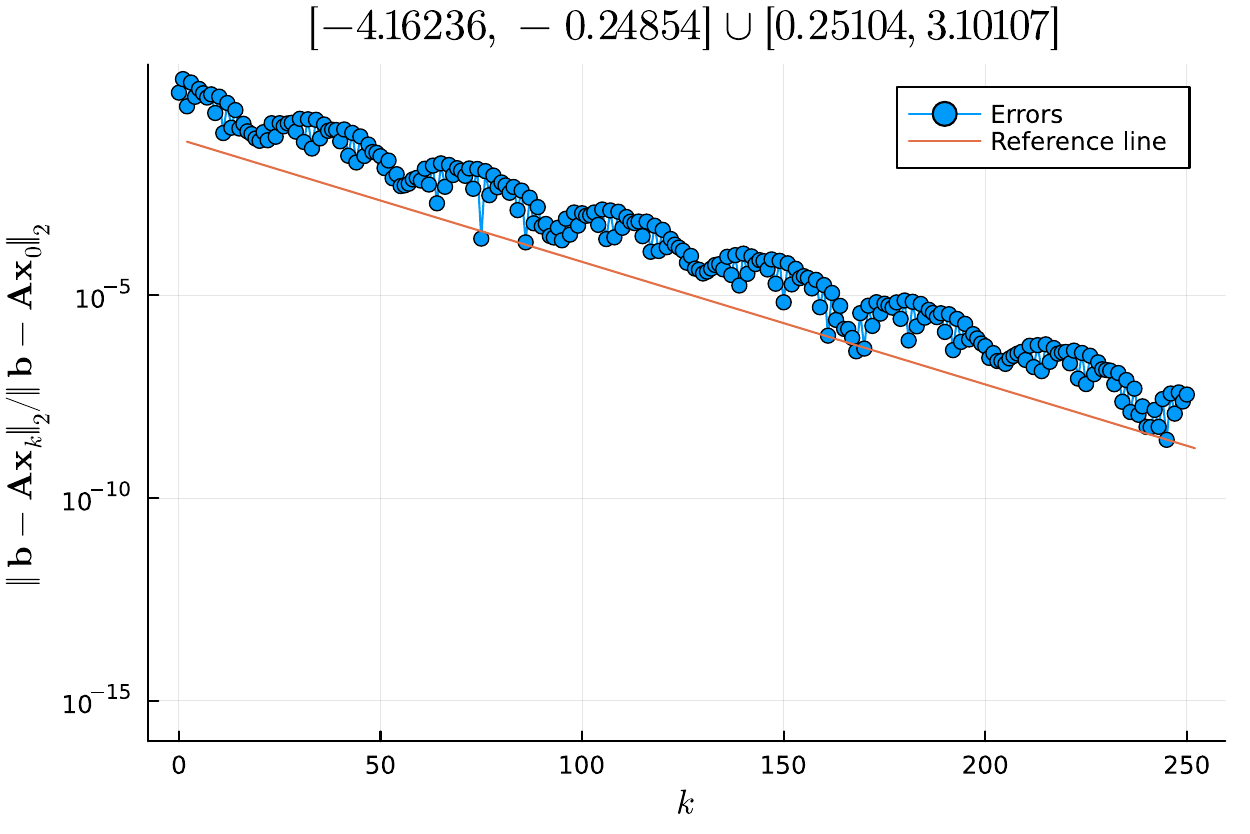}
	\end{subfigure}
	\begin{subfigure}{0.495\linewidth}
		\centering
		\includegraphics[width=\linewidth]{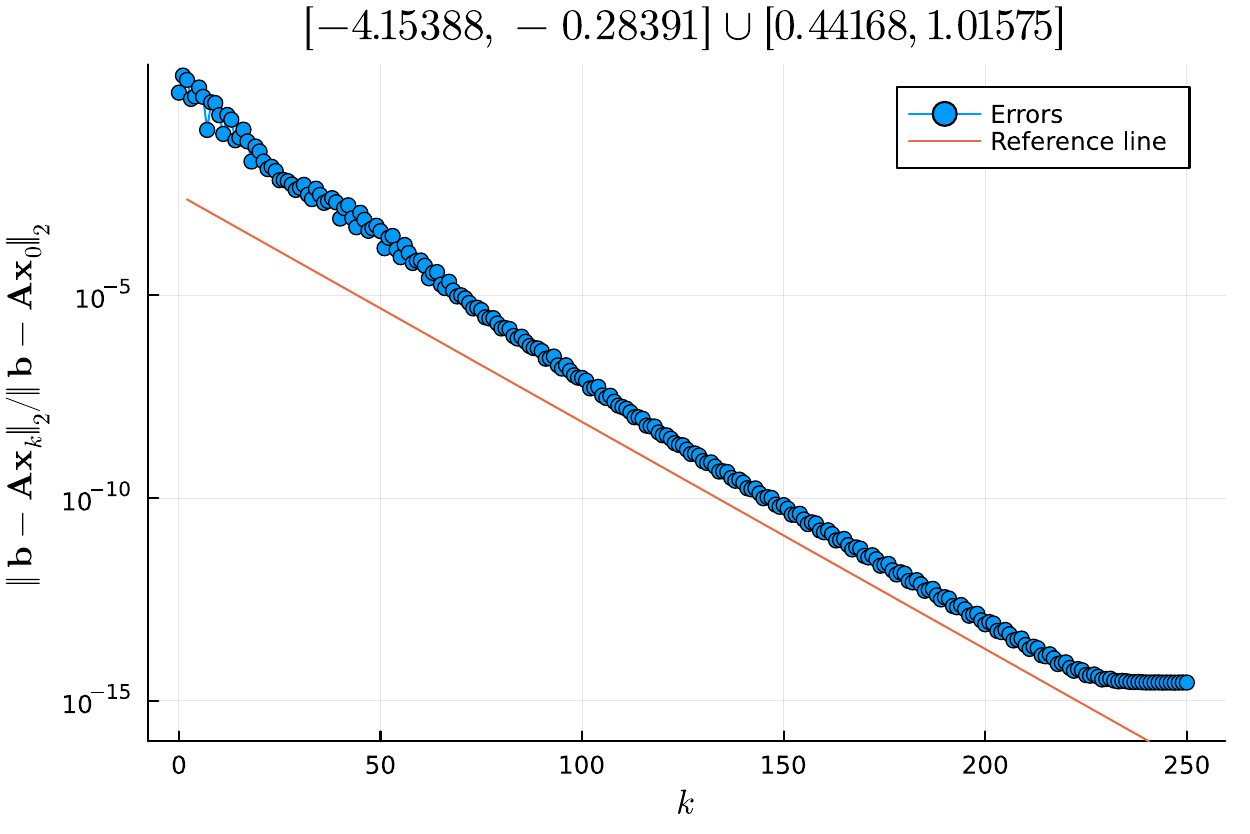}
	\end{subfigure}
	\caption{Relative residuals at each iteration for 250 iterations of the Akhiezer iteration applied to a preconditioned BVP for two different approaches for adaptive band selection (left: Algorithm~\ref{alg:adapt}, right:  Algorithm~\ref{alg:adapt}, moving one endpoint at a time).  The convergence rate $\ex^{\nu(0;\bA)}$ plotted in the reference line is approximately 0.933 (left) and 0.879 (right).}
	\label{schro}
\end{figure}

\section{Discussion and open questions}\label{sect:discuss}

First, in the current work, we have not addressed the question of adaptivity for matrices with eigenvalues off the real axis. In \cite{manteuffel_adaptive_1978}, this is handled for the classical Chebyshev iteration by estimating the convex hull of the spectrum via the power method and choosing the iteration parameters based on the known ``optimal ellipse'' containing this convex hall \cite{Manteuffel1977}. For the Akhiezer iteration, similar eigenvalue estimations could be employed, but one would instead need to look for at least two distinct regions of eigenvalues rather than the convex hull. Extending the notion of the optimal ellipse to a higher genus seems to be the primary difficulty. That is, given ellipses containing the spectrum, finding the intervals that optimize the convergence rate of the Akhiezer iteration requires a generalization of the results of \cite{Manteuffel1977} to more complex geometries based on the level curves of $\re \mathfrak g(z)$. How to best approach this is an important and interesting question that we leave for future work.  

On the other hand, the adaptivity techniques for symmetric matrices of Section~\ref{sect:adapt} can be used for eigenvalue approximation, even when decoupled from the Akhiezer iteration. When the eigenvalues are contained in two disjoint intervals, the variant based on inner products computes both interior and exterior eigenvalues to full precision. In particular, given any symmetric matrix, one could compute interior eigenvalues in the interval $[b,c]$ by choosing an initial guess of $[a,b]\cup[c,d]$. Then, the algorithm would compute two consecutive eigenvalues in $[b,c]$ or leave $b$ and/or $c$ unchanged if there are zero/one eigenvalues in this interval. This requires at most 4 Rayleigh quotient computations, so this provides an efficient method to compute interior eigenvalues to full precision.

\subsection{Other applications to orthogonal polynomials}

Through the framework established here and in \cite{part1}, an approach similar to that in \cite{Townsend2014} could be employed to compute the $N$-point Gaussian quadrature rule for a measure with density of the form \eqref{rhp_weight} in $\OO(N)$ complexity.  In particular, Akhiezer's formulae should enable this approach to yield a particularly efficient algorithm for the Akhiezer polynomials on two intervals. Unfortunately, a good application for this quadrature rule is elusive at this time. 

Finally, the computations in this work could potentially be improved upon by directly using the asymptotic formulae for the orthogonal polynomials corresponding to weight functions of the form \eqref{rhp_weight}. These asymptotic formulae could be obtained via an approach similar to that in \cite{Huybrechs2016}, and then, like the Akhiezer polynomials on two intervals, computed efficiently by simply evaluating the formulae \`a la \cite{Bogaert2014a,Hale2012}. Then, our Riemann--Hilbert-based numerical method would be required only for low degree polynomials before switching to evaluating the asymptotic formulae when the degree is sufficiently large. Such an approach will be explored further in future work.

\section*{Acknowledgements}

This work was supported in part by NSF DMS-1945652, DMS-2306438 (TT).  We would also like to thank Alex Townsend and Walter Van Assche for pointing us to valuable references. 

\appendix

\section{Computing orthogonal polynomials via a  Riemann--Hilbert problem}\label{ap:rhp}
Recurrence coefficients and Cauchy integrals of weight functions of the form \eqref{special_weight} can be computed by numerically solving a deformed version of the Fokas--Its--Kitaev (FIK) Riemann--Hilbert problem \cite{Fokas1992}. We first define the notation necessary for this problem. 

Let $\mathfrak g$ denote the exterior Green's function with pole at infinity for $\Sigma=\bigcup_{j=1}^{g+1}[a_j,b_j]$. More specifically, define 
\begin{equation*}
	\mathfrak{g}'(z)=\frac{Q_g(z)}{R(z)},\quad\text{where}~R(z)^2=\prod_{j=1}^{g+1}(z-a_j)(z-b_j).
\end{equation*}
Here, $Q_g$ is a monic polynomial of degree $g$
\begin{equation*}
	Q_g(z)=z^g+\sum_{k=0}^{g-1}h_kz^k,
\end{equation*}
whose coefficients $h_k$ are determined by the linear system 
\begin{equation*}
	\int_{b_j}^{a_{j+1}}\sum_{k=0}^{g-1}h_k\frac{z^k}{R(z)}\df z=-\int_{b_j}^{a_{j+1}}\frac{z^g}{R(z)}\df z,
\end{equation*}
for $j=1,\ldots,g$. Then, we define  $\mathfrak{g}$ by normalizing at the leftmost endpoint. If some function satisfies $\psi'=\mathfrak{g}'$, then we let 
\begin{equation*}
	\mathfrak{g}(z)=\psi(z)-\psi(a_1).
\end{equation*}
Define $\mathfrak{g}_1$ to be the coefficient in the $\OO(z^{-1})$ term in asymptotic expansion of $\mathfrak g$. That is,
\begin{equation}\label{gasymp}
\mathfrak{g}(z)=\log\mathfrak{c} z+\frac{\mathfrak{g}_1}{z}+\ldots,\quad z\to\infty.
\end{equation}
We also define
\begin{equation*}
	\Delta_j=\mathfrak{g}^+(z)-\mathfrak{g}^-(z)=2\sum_{k=1}^{j}\int_{a_k}^{b_k}(\mathfrak{g}')^+(z)\df z\in\im\real,
\end{equation*}
for $z\in(b_j,a_{j+1})$, $j=1,\ldots,g$, where this value is independent of $z$ within each of these intervals.

Define the function $\mathfrak h_n$ by
\begin{equation*}
	\begin{aligned}
		\mathfrak h_n(z)&=R(z)\left(\sum_{j=1}^{g+1}
		\frac{1}{2\pi \im }\int_{a_j}^{b_j}\frac{A_j(n)}{R_+(s)(s-z)}\df s+\sum_{\ell=1}^{g}\int_{b_\ell}^{a_{\ell+1}}\frac{\log(\ex^{n\Delta_\ell})}{R(s)(s-z)}\df s\right),
	\end{aligned}
\end{equation*}
where $A_1(n),\ldots,A_{g+1}(n)$ are determined by the linear system
\begin{equation*}
	\sum_{j=1}^{g+1}
	A_j(n)\int_{a_j}^{b_j}\frac{s^{k-1}}{R_+(s)}\df s=-\sum_{\ell=1}^{g}\log(\ex^{n\Delta_\ell})\int_{b_\ell}^{a_{\ell+1}}\frac{s^{k-1}}{R(s)}\df s,
\end{equation*}
for $k=1,\ldots,g+1$. Define $\mathfrak h_n^{(1)}$ by
\begin{equation*}
	\mathfrak h_n(z)=\frac{\mathfrak h^{(1)}_{n}}{z}+\OO(z^{-2}),\quad z\to\infty.
\end{equation*}
Finally, for a general weight function of the form \eqref{rhp_weight}, define 
\begin{equation*}
	w_j(x)=h_j(x)\left(\sqrt{x-a_j}\right)^{\alpha_j}\left(\sqrt{b_j-x}\right)^{\beta_j}.
\end{equation*}
Weight functions of the form \eqref{special_weight} satisfy $w_j=w_0|_{(a_j,b_j)}$ for all $j=1,\ldots,g+1$, for an analytic function $w_0: \mathbb C \setminus \overline{(\mathbb R \setminus \Sigma)} \to \mathbb C $.

The deformed FIK Riemann--Hilbert problem for the unknown $\bT_n:\compl\setminus\Sigma\to\compl^{2\times2}$ is 
\begin{equation}\label{eq:RHPsmall}
	\begin{aligned}
		&\bT_n~\text{is analytic in}~\compl\setminus\Sigma,\\
		&\bT_n^+(z)=\bT_n^-(z)\begin{pmatrix}
			0 &w_0(z)\ex^{-A_j(n)}\\-\frac{1}{w_0(z)}\ex^{A_j(n)} &0
		\end{pmatrix},\quad z\in(a_j,b_j),\\
		&\bT_n(z)=\bI+\OO(z^{-1}),\quad z\to\infty,\\
		&\bT_n(z)=\OO\begin{pmatrix}
			1+|z-a_j|^{-\alpha_j/2} &1+|z-a_j|^{\alpha_j/2}\\
			1+|z-a_j|^{-\alpha_j/2} &1+|z-a_j|^{\alpha_j/2}
		\end{pmatrix},\quad z\to a_j,\\
		&\bT_n(z)=\OO\begin{pmatrix}
			1+|z-b_j|^{-\beta_j/2} &1+|z-b_j|^{\beta_j/2}\\
			1+|z-b_j|^{-\beta_j/2} &1+|z-b_j|^{\beta_j/2}
		\end{pmatrix},\quad z\to b_j.
	\end{aligned}
\end{equation}
This Riemann--Hilbert problem is solved numerically via a Chebyshev collocation method, similar to that which was used in \cite{Bilman2022}. The solution is represented in shifted Chebyshev polynomial bases on each interval with singularity structure chosen to match the endpoint asymptotics. The jump conditions are enforced at each collocation point by computing the  boundary values of the Cauchy integrals of the basis polynomials, and the expansion coefficients are approximated by the solution to a block linear system.  A general reference for this type of approach is \cite{TrogdonSOBook}. 

Given a solution to \eqref{eq:RHPsmall}, the recurrence coefficients are recovered by
\begin{equation}\label{realrecurr}
	\begin{aligned}
		a_n&=\left(\bT_{n}^{(1)}\right)_{11}-\left(\bT_{n+1}^{(1)}\right)_{11}-\mathfrak h_n^{(1)}+\mathfrak h_{n+1}^{(1)}-\mathfrak g_1,\\
		b_n&=\sqrt{\left(\bT_{n+1}^{(1)}\right)_{12}\left(\bT_{n+1}^{(1)}\right)_{21}}.
	\end{aligned}
\end{equation}
Methods to numerically compute the first order correction terms and the constants $A_j(n)$ can be found in \cite[Section 4]{part1}. Furthermore, if $p_0,p_1,\ldots$ denote the orthonormal polynomials with respect to the normalization of \eqref{special_weight}, the Cauchy integrals can be computed as \cite[Equation 31]{part1}
\begin{equation}\label{cauchyints}
	\mathcal{C}_\Sigma\left[p_nw\right](z)=\left(\prod_{j=0}^{n-1}\frac{1}{b_j\mathfrak{c}}\right)\left(\bT_n(z)\right)_{12}\ex^{\mathfrak{h}_n(z)-n\mathfrak{g}(z)},
\end{equation}
where $\mathfrak c$ is as defined in \eqref{gasymp}. Methods to numerically compute $\mathfrak{c}$, $\mathfrak{g}$, and $\mathfrak{h}_n$ are found in \cite[Section 4]{part1}.

To see why it suffices to solve such a simple Riemann--Hilbert problem, consider the final deformed FIK Riemann--Hilbert problem from \cite{part1} for the unknown $\Tilde\bS_n$:
\begin{equation}\label{RHPfinal}
	\begin{aligned}
		&\Tilde \bS_n~\text{is analytic in}~\compl\setminus\left(\bigcup_{j=1}^{g+1}[a_j,b_j]\cup\bigcup_{j=1}^{g+1}C_j\right),\\
		&\Tilde \bS_n^+(z)=\Tilde \bS_n^-(z)\begin{pmatrix}
			1 &0\\-\frac{1}{w_j(z)}\ex^{2\mathfrak h_n(z)-2n\mathfrak{g}(z)} &1
		\end{pmatrix},\quad z\in C_j\cap\compl^+,\\
		&\Tilde \bS_n^+(z)=\Tilde \bS_n^-(z)\begin{pmatrix}
			1 &0\\\frac{1}{w_j(z)}\ex^{2\mathfrak h_n(z)-2n\mathfrak{g}(z)} &1
		\end{pmatrix},\quad z\in C_j\cap\compl^-,\\
		&\Tilde \bS_n^+(z)=\Tilde \bS_n^-(z)\begin{pmatrix}
			0 &w_j(z)\ex^{-A_j(n)}\\-\frac{1}{w_j(z)}\ex^{A_j(n)} &0
		\end{pmatrix},\quad z\in(a_j,b_j),\\
		&\Tilde \bS_n(z)=\bI+\OO(z^{-1}),\quad z\to\infty,\\
		&\Tilde\bS_n(z)=\OO\begin{pmatrix}
			1+|z-a_j|^{-\alpha_j/2} &1+|z-a_j|^{\alpha_j/2}\\
			1+|z-a_j|^{-\alpha_j/2} &1+|z-a_j|^{\alpha_j/2}
		\end{pmatrix},\quad z\to a_j,\\
		&\Tilde\bS_n(z)=\OO\begin{pmatrix}
			1+|z-b_j|^{-\beta_j/2} &1+|z-b_j|^{\beta_j/2}\\
			1+|z-b_j|^{-\beta_j/2} &1+|z-b_j|^{\beta_j/2}
		\end{pmatrix},\quad z\to b_j,
	\end{aligned}
\end{equation}
where each $C_j$ is a counterclockwise-oriented curve lying in $\Omega_j$ that encircles, but does not intersect $[a_j,b_j]$. Because weight functions of the form \eqref{special_weight} satisfy $w_j= w_0|_{(a_j,b_j)}$ for all $j$ and $w_0$ only has singularities at the endpoints of intervals, one can instead define $C$ to be a counterclockwise-oriented circle with $\Sigma$ in its interior. Then, the jump conditions on $C_1,\ldots C_{g+1}$ can be replaced with a single, identical jump condition on $C$.

To see why the jump conditions on $C$ can then effectively be removed, consider a function $\bPhi$ that solves the Riemann--Hilbert problem:
\begin{equation}\label{eq:modelRHP}
	\begin{aligned}
		&\bPhi^+(z)=\bPhi^-(z)\bJ(z),\quad z\in C,\\
		&\bPhi(z)=\bI+\bPhi_1 z^{-1}+\OO(z^{-2}),\quad z\to\infty,
	\end{aligned}
\end{equation}
where $\bJ(z)$ is analytic exterior to $C$ and satisfies
\begin{equation*}
	\bJ(z)=\bI+\begin{pmatrix}
		0&0\\j_1 &0
	\end{pmatrix}z^{-1}+\begin{pmatrix}
	0&0\\\OO(z^{-2}) &0
	\end{pmatrix},\quad z\to\infty.
\end{equation*}
We allow $\bPhi$ to potentially satisfy additional jump conditions in the interior of $C$. We denote the interior of $C$ by $D$. Define
\begin{equation*}
	\bPsi(z)=\begin{cases}
		\bPhi(z) &z\in D,\\
		\bPhi(z)\bJ(z) &z\in\compl\setminus\overline D.
	\end{cases}
\end{equation*}
Then, $\bPsi(z)$ satisfies the same jump conditions as $\bPhi$ in $D$, but the jump condition on $C$ is now the identity. It has asymptotic behavior 
\[
\bPsi(z)=\bI+\left(\bPhi_1+\begin{pmatrix}
	0&0\\j_1 &0
\end{pmatrix}\right)z^{-1}+\OO(z^{-2}),\quad z\to\infty.
\]

We apply this to $\tilde \bS_n$ where $\bJ$ is the jump on $C$, after the replacement of the jumps on $C_1,\ldots,C_{g+1}$ with a single jump on $C$.  Provided the asymptotics for $\bJ$ hold as hypothesized, we can essentially ignore the jump on $C$, and we find the Riemann--Hilbert problem for $\bT_n$.

Then, we note that since the $\OO(z^{-1})$ term is all that is required to compute the recurrence coefficients (see \eqref{realrecurr}) and $a_n$ requires only the $(1,1)$ entry of this term, $a_n$ is directly recovered from the simplified problem for $\bT_n$ provided that
\begin{equation*}
\frac{1}{w_0(z)}\ex^{2\mathfrak h_n(z)-2n\mathfrak{g}(z)}=\OO(z^{-1}),\quad z\to\infty.
\end{equation*}
The recovery of $b_n$ requires the additional assumption that $j_1=0$, i.e., $\bJ(z)=\bI+\OO(z^{-2})$ as $z\to\infty$.
For weight functions of the form \eqref{special_weight}, $w_0(z)=\OO(z)$ as $z\to\infty$, so 
\begin{equation*}
\frac{1}{w_0(z)}\ex^{2\mathfrak h_n(z)-2n\mathfrak{g}(z)}=\OO(z^{-1-2n}),\quad z\to\infty.
\end{equation*}
Thus, $j_1=0$ for all $n>0$. Since $b_n$ is recovered from $\bT_{n+1}$, this is sufficient to compute all the recurrence coefficients for this weight. Furthermore, since multiplication by $\bJ(z)$ does not affect the $(1,2)$ entry, it is sufficient to compute the Cauchy integrals according to \eqref{cauchyints}.

This analysis can be extended to weight functions $w$ that consist of a meromorphic function multiplied by the weight \eqref{special_weight} that are constrained to only have zeroes and poles at the endpoints of intervals. Then, the jump conditions on each $C_j$ can still be replaced with a jump condition on $C$, but whether this jump condition can be thrown away depends on the asymptotic behavior of $w$. In particular, the Akhiezer weight \eqref{eq:gen_akh_weight} satisfies $w_0(z)=\OO(z^{-1})$ as $z\to\infty$, so the full Riemann--Hilbert problem \eqref{RHPfinal} is required only to compute the $0$th pair of recurrence coefficients, and all others can be computed by solving the simplified problem \eqref{eq:RHPsmall}.

\section{Akhiezer polynomial recurrence coefficients}\label{ap:recurr}
Assume that the $n$th degree orthonormal polynomial can be represented as 
\begin{equation*}
	p_n(x)=k_nx^n+\ell_nx^{n-1}+\ldots.
\end{equation*}
Then, the recurrence coefficients can be recovered by \cite{deift_2000}
\begin{equation}\label{coeff_formula}
	a_n=\frac{\ell_n}{k_n}-\frac{\ell_{n+1}}{k_{n+1}},\quad b_n=\frac{k_n}{k_{n+1}},
\end{equation}
with the convention that $\ell_0=0,k_0=1$. To obtain such a representation, we must expand \eqref{u_formula} around $u=-\rho$, the point corresponding to $x=\infty$ on one sheet of the Riemann surface. Noting that $H$ is an odd function, we expand
\begin{align*}
	\frac{H(u-\rho)}{H(u+\rho)}&=
	\left(-H(2\rho)+H'(2\rho)(u+\rho)+\OO((u+\rho)^2)\right)\frac{1}{H'(0)(u+\rho)}\left(1+\OO((u+\rho)^2)\right)^{-1}\\&=
	-\frac{H(2\rho)}{H'(0)}\frac{1}{u+\rho}+\frac{H'(2\rho)}{H'(0)}+\OO(u+\rho).
\end{align*}
Similarly, $\Theta$ is an even function, and we expand
\begin{align*}
	\frac{\Theta(u+2n\rho)}{\Theta(u)}&=
	\left(\Theta\left((2n-1)\rho\right)+\Theta'\left((2n-1)\rho\right)(u+\rho)+\ldots\right)\frac{1}{\Theta(\rho)}\left(1-\frac{\Theta'(\rho)}{\Theta(\rho)}(u+\rho)+\ldots\right)^{-1}\\&=
	\frac{\Theta\left((2n-1)\rho\right)}{\Theta(\rho)}+\left(\frac{\Theta'\left((2n-1)\rho\right)}{\Theta(\rho)}+\frac{\Theta\left((2n-1)\rho\right)\Theta'(\rho)}{\Theta^2(\rho)}\right)(u+\rho)+\OO\left((u+\rho)^2\right).
\end{align*}
The other product summed in \eqref{poly} is $\OO(u+\rho)$ and will not contribute to the recurrence coefficients. To convert these expansions from $u+\rho$ to $x$, we utilize the relation 
\begin{equation*}
	x-\alpha=\frac{1-\alpha^2}{2(\sn^2u-\sn^2\rho)},
\end{equation*}
which we rewrite as 
\begin{equation}\label{eq:relation}
	(\sn u+\sn\rho)(\sn u-\sn\rho)=\frac{1-\alpha^2}{2(x-\alpha)}.
\end{equation}
Noting that $\sn$ is an odd function, we expand
\begin{align*}
	(\sn u+\sn\rho)(\sn u-\sn\rho)&=\left(\sn'\rho(u+\rho)-\frac{\sn''\rho}{2}(u+\rho)^2\ldots\right)\left(-2\sn\rho+\sn'\rho(u+\rho)+\ldots\right)\\&=
	-2\sn\rho\sn'\rho(u+\rho)\left(1-\left(\frac{\sn'\rho}{2\sn\rho}+\frac{\sn''\rho}{2\sn'\rho}\right)(u+\rho)+\OO\left((u+\rho)^2\right)\right).
\end{align*}
Substituting this into \eqref{eq:relation}, solving for $u+\rho$, and iterating,
\begin{align*}
	u+\rho&=-\frac{1-\alpha^2}{4\sn\rho\sn'\rho}\left(1-\left(\frac{\sn'\rho}{2\sn\rho}+\frac{\sn''\rho}{2\sn'\rho}\right)(u+\rho)+\OO\left((u+\rho)^2\right)\right)^{-1}\frac{1}{x-\alpha}\\&=
	-\frac{1-\alpha^2}{4\sn\rho\sn'\rho}\left(1-\frac{(1-\alpha^2)}{4\sn\rho\sn'\rho}\left(\frac{\sn'\rho}{2\sn\rho}+\frac{\sn''\rho}{2\sn'\rho}\right)\frac{1}{x-\alpha}+\OO\left(\left(\frac{1}{x-\alpha}\right)^2\right)\right)\frac{1}{x-\alpha}.
\end{align*}
Inverting this,
\begin{align*}
	&\frac{1}{u+\rho}=
	-\frac{4\sn\rho\sn'\rho}{1-\alpha^2}(x-\alpha)-\left(\frac{\sn'\rho}{2\sn\rho}+\frac{\sn''\rho}{2\sn'\rho}\right)+\OO\left(\frac{1}{x-\alpha}\right).
\end{align*}
Substituting these into the series in $u+\rho$,
\begin{align*}
	&\frac{H(u-\rho)}{H(u+\rho)}=\frac{H(2\rho)}{H'(0)}\frac{4\sn\rho\sn'\rho}{1-\alpha^2}(x-\alpha)+\frac{H(2\rho)}{H'(0)}\left(\frac{\sn'\rho}{2\sn\rho}+\frac{\sn''\rho}{2\sn'\rho}\right)+\frac{H'(2\rho)}{H'(0)}+\OO\left(\frac{1}{x-\alpha}\right),\\
	&\frac{\Theta(u+2n\rho)}{\Theta(u)}=\frac{\Theta\left((2n-1)\rho\right)}{\Theta(\rho)}-\frac{1-\alpha^2}{4\sn\rho\sn'\rho}\left(\frac{\Theta'\left((2n-1)\rho\right)}{\Theta(\rho)}+\frac{\Theta\left((2n-1)\rho\right)\Theta'(\rho)}{\Theta^2(\rho)}\right)\frac{1}{x-\alpha}\\&+\OO\left(\left(\frac{1}{x-\alpha}\right)^2\right).
\end{align*}
To obtain the recurrence coefficients from these expansions, denote 
\begin{align*}
	c&=\frac{H(2\rho)}{H'(0)}\frac{4\sn\rho\sn'\rho}{1-\alpha^2},\quad &d&=\frac{H(2\rho)}{H'(0)}\left(\frac{\sn'\rho}{2\sn\rho}+\frac{\sn''\rho}{2\sn'\rho}\right)+\frac{H'(2\rho)}{H'(0)},\\
	\gamma_n&=\frac{\Theta\left((2n-1)\rho\right)}{\Theta(\rho)},&\delta_n&=-\frac{1-\alpha^2}{4\sn\rho\sn'\rho}\left(\frac{\Theta'\left((2n-1)\rho\right)}{\Theta(\rho)}+\frac{\Theta\left((2n-1)\rho\right)\Theta'(\rho)}{\Theta^2(\rho)}\right).
\end{align*}
Then,
\begin{align*}
	p_n(x)&=C_n\left(c(x-\alpha)+d+\ldots\right)^n\left(\gamma_n+\frac{\delta_n}{x-\alpha}+\ldots\right)+\OO\left(\frac{1}{x-\alpha}\right)\\&=
	C_n\gamma_nc^nx^n+C_n(n\gamma_nc^{n-1}d+\delta_nc^n-n\alpha\gamma_nc^n)x^{n-1}+\ldots,
\end{align*}
so
\begin{equation*}
	k_n=C_n\gamma_nc^n,\quad \ell_n=C_n(n\gamma_nc^{n-1}d+\delta_nc^n-n\alpha\gamma_nc^n).
\end{equation*}
Applying \eqref{coeff_formula}, we get that
\begin{align*}
	&a_n=\frac{\ell_n}{k_n}-\frac{\ell_{n+1}}{k_{n+1}}=\left(\frac{nd}{c}+\frac{\delta_n}{\gamma_n}-n\alpha\right)-\left(\frac{(n+1)d}{c}+\frac{\delta_{n+1}}{\gamma_{n+1}}-(n+1)\alpha\right)=-\frac{d}{c}+\frac{\delta_n}{\gamma_n}-\frac{\delta_{n+1}}{\gamma_{n+1}}+\alpha,\\
	&b_n=\frac{k_n}{k_{n+1}}=\frac{C_n}{C_{n+1}}\frac{1}{c},
\end{align*}
for $n\geq1$ and 
\begin{align*}
	&a_0=-\frac{\ell_1}{k_1}=-\frac{d}{c}-\frac{\delta_1}{\gamma_1}+\alpha,\\
	&b_0=\frac{2}{C_1}\frac{1}{c}.
\end{align*}
Finally, we substitute in $c,d,\gamma_n,\delta_n$ and $C_n$ and get
\begin{align*}
	a_0=&(1-\alpha^2)\left(\frac{1}{8\sn^2\rho}+\frac{\sn''\rho}{8\sn\rho\left(\sn'\rho\right)^2}+\frac{H'(2\rho)}{H(2\rho)}\frac{1}{4\sn\rho\sn'\rho}\right)+\frac{1-\alpha^2}{2\sn\rho\sn'\rho}\frac{\Theta'\left(\rho\right)}{\Theta\left(\rho\right)}+\alpha,\\
	a_n=&(1-\alpha^2)\left(\frac{1}{8\sn^2\rho}+\frac{\sn''\rho}{8\sn\rho\left(\sn'\rho\right)^2}\right)\\&+\frac{1-\alpha^2}{4\sn\rho\sn'\rho}\left(\frac{H'(2\rho)}{H(2\rho)}+\frac{\Theta'\left((2n+1)\rho\right)}{\Theta\left((2n+1)\rho\right)}-\frac{\Theta'\left((2n-1)\rho\right)}{\Theta\left((2n-1)\rho\right)}\right)+\alpha,\quad n\geq1,\\
	b_0=&\frac{1-\alpha^2}{4\sn\rho\sn'\rho}\frac{H'(0)}{H(2\rho)}\sqrt{\frac{2\Theta(3\rho)}{\Theta(\rho)}},\\
	b_n=&\frac{1-\alpha^2}{4\sn\rho\sn'\rho}\frac{H'(0)}{H(2\rho)}\frac{\sqrt{\Theta((2n+3)\rho)\Theta((2n-1)\rho)}}{\Theta((2n+1)\rho)},\quad n\geq1.
\end{align*}

\section{The classical Chebyshev iteration and our modification}\label{ap:cheb_iter}
To compare our presentation of the Chebyshev iteration with that of \cite{gutknecht_chebyshev_2002}, suppose that the interval $[\alpha-c,\alpha+c]$ does not contain the origin, and for simplicity suppose $\alpha - c > 0$. Define the shifted and scaled Chebyshev-$T$ polynomials by
\[
\Tilde T_k(x)=T_k\left(\frac{x-\alpha}{c}\right).
\]
Based on this, define
\begin{equation*}
r_k(x)=\frac{\Tilde  T_k(x)}{\Tilde T_k(0)},\quad q_k(x)=\frac{1-r_k(x)}{x}.
\end{equation*}
We have the recurrence
\begin{equation}\label{eq:shift_cheb_recurr}
\begin{aligned}
\frac{x-\alpha}{c}\Tilde T_0(x)&=\Tilde T_1(x),\\
2\frac{x-\alpha}{c}\Tilde T_k(x)&=\Tilde T_{k-1}(x)+\Tilde T_{k+1}(x),\quad k\geq1,
\end{aligned}
\end{equation}
so $r_k$ must satisfy 
\begin{equation*}
	\begin{aligned}
		\frac{x-\alpha}{c}r_0(x)&=r_1(x)\frac{\Tilde T_1(0)}{\Tilde T_0(0)},\\
		2\frac{x-\alpha}{c}r_k(x)&=r_{k-1}(x)\frac{\Tilde T_{k-1}(0)}{\Tilde T_{k}(0)}+r_{k+1}(x)\frac{\Tilde T_{k+1}(0)}{\Tilde T_{k}(0)},\quad k\geq1.
	\end{aligned}
\end{equation*}
Using $r_k(x)=1-xq_k(x)$ and applying \eqref{eq:shift_cheb_recurr},
\begin{equation*}
	\begin{aligned}
		\frac{x-\alpha}{c}(1-xq_0(x))&=\frac{\alpha}{c}xq_1(x)-\frac{\alpha}{c},\\
		2\frac{x-\alpha}{c}(1-xq_k(x))&=-xq_{k-1}(x)\frac{\Tilde T_{k-1}(0)}{\Tilde T_{k}(0)}-xq_{k+1}(x)\frac{\Tilde T_{k+1}(0)}{\Tilde T_{k}(0)}-\frac{2\alpha}{c},\quad k\geq1.
	\end{aligned}
\end{equation*}
This can be rearranged as
\begin{equation*}
	\begin{aligned}
		q_1(x)&=\frac{1}{\alpha}(1-xq_0(x)+\alpha q_0(x)),\\
		q_{k+1}(x)&=-\frac{\Tilde T_{k}(0)}{\Tilde T_{k+1}(0)}\left(\frac{2}{c}\left(1-xq_k(x)+\alpha q_k(x)\right)+\frac{\Tilde T_{k-1}(0)}{\Tilde T_{k}(0)}q_{k-1}(x)\right),\quad k\geq1.
	\end{aligned}
\end{equation*}
Thus, 
\begin{equation*}
q_{k+1}(x)=-\frac{1}{\gamma_k}(r_k(x)+\alpha q_k(x)+\beta_{k-1} q_{k-1}(x)),\quad k\geq0,
\end{equation*}
where 
\begin{equation*}
\begin{aligned}
&\gamma_0=-\frac{1}{\alpha},\quad &\gamma_k&=\frac{c}{2}\frac{\Tilde T_{k+1}(0)}{\Tilde T_{k}(0)}=\frac{c}{2}\frac{T_{k+1}(-\alpha/c)}{ T_{k}(-\alpha/c)},\quad k\geq1,\\
&\beta_{-1}=0,\quad &\beta_k&=\frac{c}{2}\frac{\Tilde T_{k}(0)}{\Tilde T_{k+1}(0)}=\frac{c}{2}\frac{ T_{k}(-\alpha/c)}{T_{k+1}(-\alpha/c)},\quad k\geq0.
\end{aligned}
\end{equation*}
Furthermore, the recurrence for $r_k$ can be rearranged as
\begin{equation*}
r_{k+1}(x)=\frac{1}{\gamma_k}(xr_k(x)-\alpha r_k(x)-\beta_{k-1} r_{k-1}(x)),\quad k\geq0.
\end{equation*}

\subsection{Error bounds for the residuals}
The Chebyshev iteration as given in \cite[Algorithm 1]{gutknecht_chebyshev_2002} is recovered by replacing $q_k(x)$ with $\bx_k$ and $r_k(x)$ with $\br_k$ in these recurrences. That is, the iterates are defined by $\bx_k=q_k(\bA)\bb$, and the residuals are given by $\br_k=\bb-\bA\bx_k=r_k(\bA)\bb$ where $\br_k$ is presumably small because $r_k$ is small on $[\alpha-c,\alpha+c]$ due to the minimax property of the Chebyshev polynomials. More specifically, for $x \in [\alpha - c, \alpha + c]$,
\[
|r_k(x)|=\left|\frac{\Tilde T_k(x)}{\Tilde T_k(0)}\right|\leq\frac{2 |\rho|^k}{1 +|\rho|^{2k}}, \quad \rho = J_+^{-1}\left( - \frac{\alpha}{c} \right),
\]
where $J^{-1}_+$ is a right inverse of the Joukowsky map, $z \mapsto (z + z^{-1})/2$, given by
\[
J^{-1}_+(z)=z-\sqrt{z-1}\sqrt{z+1},
\] 
which maps $\compl \setminus [-1,1]$ to the interior of the unit disk. This results in error bounds for $\br_k$.  An interesting aspect of this bound on $r_k$ is that it remains bounded as $|\rho| \to 1$.  

In contrast, our modification of the Chebyshev iteration implemented in Algorithm~\ref{alg:cheb_iter} relies on building an approximation to the matrix inverse out of Chebyshev polynomials. Define the normalized shifted and scaled Chebyshev-$T$ polynomials by
\begin{equation*}
\hat T_0(x)=\Tilde{T}_0(x),\quad \hat T_k(x)=\sqrt{2}\Tilde{T}_k(x),\quad k\geq1.
\end{equation*}
From \cite[Section 4.1]{part1}, we have the following recurrence for $S_k=\int_{\alpha-c}^{\alpha+c}\frac{\Tilde T_j(z)}{z}\frac{\df z}{\pi\sqrt{z-\alpha+c}\sqrt{\alpha+c-z}}$:
\begin{equation*}
S_0=\frac{1}{\sqrt{\alpha-c}\sqrt{\alpha+c}},\quad S_{k}=S_{k-1}J^{-1}_+\left(\frac{-\alpha}{c}\right),\quad k\geq 1,
\end{equation*}
Then, \eqref{eq:cheby_iter} can be rewritten as
\begin{equation}\label{eq:stielt}
\bx_{k+1}=S_0\Tilde T_0(\bA)\bb+2\sum_{j=1}^{k}S_j\Tilde T_j(\bA)\bb.
\end{equation}
Coupling this with the recurrence \eqref{eq:shift_cheb_recurr} yields the implementation of our modification of the Chebyshev iteration in Algorithm~\ref{alg:cheb_iter}.

To analyze the convergence, we utilize the series
\begin{equation*}
\frac{1}{x}=\frac{2}{c}\frac{1}{\sqrt{\frac{\alpha}{c}-1}\sqrt{\frac{\alpha}{c}+1}}\sump_{j=0}^{\infty}\rho^j\Tilde T_j(x),\quad \rho=J^{-1}_+\left(\frac{-\alpha}{c}\right),
\end{equation*}
where the prime denotes that the $0$th term is halved. This follows from \eqref{eq:stielt} or \cite[Equation 5.14]{Mason2002}. Then, to estimate the error for $k$th iteration we use, for $x \in [\alpha -c, \alpha + c]$,
\begin{equation*}
\begin{aligned}
G_{k}(x)=\frac{1}{x}-\frac{2}{c}\frac{1}{\sqrt{\frac{\alpha}{c}-1}\sqrt{\frac{\alpha}{c}+1}}\sump_{j=0}^{k-1}\rho^j\Tilde T_j(x)=\frac{2}{c}\frac{1}{\sqrt{\frac{\alpha}{c}-1}\sqrt{\frac{\alpha}{c}+1}}\sum_{j=k}^{\infty}\rho^j\Tilde T_j(x).
\end{aligned}
\end{equation*}

To instead investigate the residual $\bb-\bA\bx_{k}$ of our modification, we use the three-term recurrence for $\tilde T_j$ to find
\begin{equation*}
\begin{aligned}
xG_{k}(x)=&\frac{2}{c}\frac{x}{\sqrt{\frac{\alpha}{c}-1}\sqrt{\frac{\alpha}{c}+1}}\sum_{j=k}^{\infty}\rho^j\Tilde T_j(x)
=\frac{\rho^{k-1}(\rho\Tilde T_{k-1}(x)-\Tilde T_{k}(x))}{\sqrt{\frac{\alpha}{c}-1}\sqrt{\frac{\alpha}{c}+1}} \\
&= -\frac{\rho^{k-1}(\Tilde T_{k-1}(x)+\Tilde T_{k}(x))}{\sqrt{\frac{\alpha}{c}-1}\sqrt{\frac{\alpha}{c}+1}} + \frac{\rho^{k-1}(\rho + 1)\Tilde T_{k-1}(x)}{\sqrt{\frac{\alpha}{c}-1}\sqrt{\frac{\alpha}{c}+1}}.
\end{aligned}
\end{equation*}
From the identity $T_{k}(x)+T_{k-1}(x)=(x+1)V_{k-1}(x)$ where $V_{k}$ is the $k$th Chebyshev polynomial of the third kind \cite{Mason2002} and the formula for $\rho$, it follows that
\begin{equation*}
\begin{aligned}
xG_{k}(x) &= -\frac{\rho^{k-1}\left( \frac{x - \alpha}{c} + 1 \right) \tilde V_{k-1}(x)}{\sqrt{\frac{\alpha}{c}-1}\sqrt{\frac{\alpha}{c}+1}} + \frac{\rho^{k-1}(-\frac{\alpha}{c} + 1)\Tilde T_{k-1}(x)}{\sqrt{\frac{\alpha}{c}-1}\sqrt{\frac{\alpha}{c}+1}} + \rho^{k-1}\Tilde T_{k-1}(x)\\
&=  -\frac{\rho^{k-1} x \tilde V_{k-1}(x)}{\sqrt{\alpha - c}\sqrt{\alpha + c}} + \frac{\rho^{k-1}\sqrt{\frac{\alpha}{c}-1}(\tilde V_{k-1}(x) - \Tilde T_{k-1}(x))}{\sqrt{\frac{\alpha}{c}+1}} + \rho^{k-1}\Tilde T_{k-1}(x),\\
\tilde V_k(x) &= V_k\left( \frac{x - \alpha}{c} \right).
\end{aligned}
\end{equation*}
We see that as $\alpha/c \to 1$, $\rho \to -1$, we do not enjoy the same boundedness as $r_k(x)$ above, but for $x \ll 1$, in this limit, $x G_k(x)$ remains under control.  Thus, if one cares about minimizing the residual, for an ill-conditioned linear system, the classical Chebyshev iteration is likely to perform better.

\subsection{Error bounds for the true error}
To get a sense of the error
\begin{align*}
    \be_k = \bx_k - \bA^{-1} \bb,
\end{align*}
we compute 
\begin{align*}
|G_k(x)| &\leq \frac{2|\rho|^{k}}{\left|\sqrt{\alpha-c}\sqrt{\alpha+c}\right|}\frac{1}{1-|\rho|} = \frac{2|\rho|^{k}}{\left|\sqrt{\alpha-c}\sqrt{\alpha+c}\right|}\frac{1}{1 - \frac{\alpha}{c} - \sqrt{\frac{\alpha}{c}-1}\sqrt{\frac{\alpha}{c}+1}}\\
& = \frac{2 c |\rho|^{k}}{\left|\alpha-c\right| \left| \sqrt{\alpha+c}\right|}\frac{1}{|\sqrt{\alpha + c}| - |\sqrt{c - \alpha}|}.
\end{align*}
The coefficient of $|\alpha - c|^{-1}|\rho|^k$ tends to one as $\alpha/c \to 1$. 

Then the corresponding quantity for the classical Chebyshev iteration is
\begin{align*}
\left| \frac{1}{x} - q_k(x)\right| = \left| \frac{\tilde T_k(x)}{x\tilde T_k(0)} \right| \leq \frac{2 |\rho|^k}{|x|(1 +|\rho|^{2k})},
\end{align*}
which is maximized when $x = \alpha - c$ giving
\begin{align*}
    \left| \frac{1}{x} - q_k(x)\right| \leq  \frac{2 |\rho|^k}{|\alpha - c|(1 +|\rho|^{2k})}.
\end{align*}
Here the coefficient of $|\alpha - c|^{-1}|\rho|^k$ also tends to one as $\alpha/c \to 1$.
Thus, the true error in our modification of the Chebyshev iteration performs similarly to that of the classical Chebyshev iteration in this limit when the system is ill-conditioned.

\bibliographystyle{plain}
\bibliography{refs}

\begin{thebibliography}{10}

\bibitem{akhiezer}
N.~I. Akhiezer.
\newblock {\em {Elements of the Theory of Elliptic Functions}}, volume~79 of {\em Translations of Mathematical Monographs}.
\newblock American Mathematical Society, 1990.

\bibitem{code_repo}
C.~Ballew and T.~Trogdon.
\newblock \url{https://github.com/cade-b/AkhiezerIteration}, 2023.

\bibitem{part1}
C.~Ballew and T.~Trogdon.
\newblock {A Riemann–Hilbert approach to computing the inverse spectral map for measures supported on disjoint intervals}.
\newblock {\em Studies in Applied Mathematics}, 152(1):31--72, 2024.

\bibitem{Bilman2022}
D.~Bilman, P.~Nabelek, and T.~Trogdon.
\newblock {Computation of large-genus solutions of the Korteweg–de Vries equation}.
\newblock {\em Physica D: Nonlinear Phenomena}, 449:133715, July 2023.

\bibitem{Bogaert2014a}
I.~Bogaert.
\newblock {Iteration-Free Computation of Gauss--Legendre Quadrature Nodes and Weights}.
\newblock {\em SIAM Journal on Scientific Computing}, 36(3):A1008--A1026, January 2014.

\bibitem{SaadMatrixFunction}
J.~Chen, M.~Anitescu, and Y.~Saad.
\newblock {Computing {$f(A)b$} via Least Squares Polynomial Approximations}.
\newblock {\em SIAM Journal on Scientific Computing}, 33(1):195--222, 2011.

\bibitem{Chen2007}
Y.~Chen and A.~R. Its.
\newblock {A Riemann{\textendash}Hilbert approach to the Akhiezer polynomials}.
\newblock {\em Philosophical Transactions of the Royal Society A: Mathematical, Physical and Engineering Sciences}, 366(1867):973--1003, 2007.

\bibitem{Chen2002}
Y.~Chen and N.~Lawrence.
\newblock {A generalization of the Chebyshev polynomials}.
\newblock {\em J. Phys. A: Math. Gen}, 35:4651--4699, 2002.

\bibitem{de_boor_extremal_1982}
C.~de~Boor and J.~R. Rice.
\newblock Extremal {Polynomials} with {Application} to {Richardson} {Iteration} for {Indefinite} {Linear} {Systems}.
\newblock {\em SIAM Journal on Scientific and Statistical Computing}, 3(1):47--57, March 1982.

\bibitem{deift_2000}
P.~Deift.
\newblock {\em Orthogonal polynomials and random matrices: A Riemann-Hilbert approach}.
\newblock American Math. Soc., 2000.

\bibitem{Ding2022}
X.~Ding and T.~Trogdon.
\newblock {A Riemann–Hilbert Approach to the Perturbation Theory for Orthogonal Polynomials: Applications to Numerical Linear Algebra and Random Matrix Theory}.
\newblock {\em International Mathematics Research Notices}, 2023.

\bibitem{NIST:DLMF}
{\it NIST Digital Library of Mathematical Functions}.
\newblock \url{https://dlmf.nist.gov/}, Release 1.1.11 of 2023-09-15.
\newblock F.~W.~J. Olver, A.~B. {Olde Daalhuis}, D.~W. Lozier, B.~I. Schneider, R.~F. Boisvert, C.~W. Clark, B.~R. Miller, B.~V. Saunders, H.~S. Cohl, and M.~A. McClain, eds.

\bibitem{cheby_iter1950}
D.~A. Flanders and G.~Shortley.
\newblock {Numerical Determination of Fundamental Modes}.
\newblock {\em Journal of Applied Physics}, 21(12):1326--1332, 1950.

\bibitem{Fokas1992}
A.~S. Fokas, A.~R. Its, and A.~V. Kitaev.
\newblock {The isomonodromy approach to matric models in 2D quantum gravity}.
\newblock {\em Communications in Mathematical Physics}, 147(2):395--430, July 1992.

\bibitem{gautschi}
W.~Gautschi.
\newblock {\em {Orthogonal Polynomials: Computation and Approximation}}.
\newblock Oxford University Press, 2004.

\bibitem{Gragg1984}
W.~B. Gragg and W.~J. Harrod.
\newblock {The numerically stable reconstruction of Jacobi matrices from spectral data}.
\newblock {\em Numerische Mathematik}, 44(3):317--335, 1984.

\bibitem{gutknecht_chebyshev_2002}
M.~H. Gutknecht and S.~Röllin.
\newblock The {Chebyshev} iteration revisited.
\newblock {\em Parallel Computing}, 28(2):263--283, February 2002.

\bibitem{Hale2012}
N.~Hale and A.~Townsend.
\newblock {Fast and Accurate Computation of Gauss-Legendre and Gauss-Jacobi Quadrature Nodes and Weights}.
\newblock {\em SIAM J. Sci. Comput.}, 35(2):A652--A674, 2013.

\bibitem{Huybrechs2016}
D.~Huybrechs and P.~Opsomer.
\newblock {Construction and implementation of asymptotic expansions for Laguerre-type orthogonal polynomials}.
\newblock {\em IMA Journal of Numerical Analysis}, 38(3):1085--1118, December 2018.

\bibitem{Manteuffel1977}
T.~A. Manteuffel.
\newblock {The Tchebychev iteration for nonsymmetric linear systems}.
\newblock {\em Numerische Mathematik}, 28(3):307--327, 1977.

\bibitem{manteuffel_adaptive_1978}
T.~A. Manteuffel.
\newblock Adaptive procedure for estimating parameters for the nonsymmetric {Tchebychev} iteration.
\newblock {\em Numerische Mathematik}, 31(2):183--208, June 1978.

\bibitem{Mason2002}
J.~C. Mason and D.~C. Handscomb.
\newblock {\em {Chebyshev Polynomials}}.
\newblock Chapman and Hall/{CRC}, 2002.

\bibitem{Nakatsukasa2018}
Y.~Nakatsukasa, O.~S{\`{e}}te, and L.~N. Trefethen.
\newblock {The AAA Algorithm for Rational Approximation}.
\newblock {\em SIAM Journal on Scientific Computing}, 40(3):A1494--A1522, January 2018.

\bibitem{olver_slevinsky_townsend_2020}
S.~Olver, R.~M. Slevinsky, and A.~Townsend.
\newblock {Fast algorithms using orthogonal polynomials}.
\newblock {\em Acta Numerica}, 29:573–699, 2020.

\bibitem{TrogdonSORMT}
S.~Olver and T.~Trogdon.
\newblock {Numerical solution of Riemann--Hilbert Problems: Random matrix theory and orthogonal polynomials}.
\newblock {\em Constructive Approximation}, 39(1):101--149, December 2013.

\bibitem{rivlin_81}
T.~J. Rivlin.
\newblock {\em An introduction to the approximation of functions}.
\newblock Dover books on advanced mathematics. Dover, unabridged and corr. republication of the 1969 ed edition, 1981.

\bibitem{Saad}
Y.~Saad.
\newblock {Iterative Solution of Indefinite Symmetric Linear Systems by Methods Using Orthogonal Polynomials over Two Disjoint Intervals}.
\newblock {\em SIAM Journal on Numerical Analysis}, 20(4):784--811, 1983.

\bibitem{Szego1939}
G.~Szegő.
\newblock {\em {Orthogonal Polynomials}}, volume~23 of {\em Colloquium Publications}.
\newblock American Mathematical Society, Providence, Rhode Island, December 1939.

\bibitem{Townsend2014}
A.~Townsend, T.~Trogdon, and S.~Olver.
\newblock {Fast computation of Gauss quadrature nodes and weights on the whole real line}.
\newblock {\em IMA Journal of Numerical Analysis}, 36(1):337--358, October 2014.

\bibitem{trapezoid}
L.~N. Trefethen and J.~A.~C. Weideman.
\newblock {The Exponentially Convergent Trapezoidal Rule}.
\newblock {\em SIAM Review}, 56(3):385--458, 2014.

\bibitem{Trogdon2013b}
T.~Trogdon and S.~Olver.
\newblock {A Riemann-Hilbert approach to Jacobi operators and Gaussian quadrature}.
\newblock {\em IMA Journal of Numerical Analysis}, 36(1):174--196, January 2014.

\bibitem{TrogdonSOBook}
T.~Trogdon and S.~Olver.
\newblock {\em {Riemann--Hilbert Problems, Their Numerical Solution and the Computation of Nonlinear Special Functions}}.
\newblock SIAM, Philadelphia, PA, 2016.

\end{thebibliography}
\end{document}